\documentclass[reqno]{amsart}
\usepackage{amssymb}
\usepackage{amsmath}
\usepackage{enumerate}
\usepackage{srcltx}
\usepackage[mathscr]{euscript}
\usepackage{times}
\usepackage{color}

\makeatletter
\@addtoreset{equation}{section}
\makeatother

\renewcommand\thefigure{\thesection.\@arabic\c@figure}
\renewcommand\thetable{\thesection.\@arabic\c@table}

\newtheorem{theorem}{Theorem}[section]
\newtheorem{lemma}[theorem]{Lemma}
\newtheorem{proposition}[theorem]{Proposition}
\newtheorem{corollary}[theorem]{Corollary}

\newtheorem{remark}[theorem]{Remark}

\newcounter{as}

\newcommand{\mc}[1]{{\mathcal #1}}
\newcommand{\mf}[1]{{\mathfrak #1}}
\newcommand{\mb}[1]{{\mathbf #1}}
\newcommand{\bb}[1]{{\mathbb #1}}
\newcommand{\bs}[1]{{\boldsymbol #1}}
\newcommand{\ms}[1]{{\mathscr #1}}
\newcommand{\dom}{\ms D_0}

\title{From coalescing random walks on a torus to Kingman's
  coalescent}

\author{J. Beltr\'an, E. Chavez, C. Landim}

\address{\noindent IMCA, Calle los Bi\'ologos 245, Urb. San C\'esar
  Primera Etapa, Lima 12, Per\'u and PUCP, Av. Universitaria cdra. 18,
  San Miguel, Ap. 1761, Lima 100, Per\'u.  \newline e-mail: \rm
  \texttt{johel.beltran@pucp.edu.pe} }

\address{\noindent PUCP, Av. Universitaria cdra. 18,
  San Miguel, Ap. 1761, Lima 100, Per\'u.  \newline e-mail: \rm
  \texttt{johel.beltran@pucp.edu.pe} }

\address{\noindent IMPA, Estrada Dona Castorina 110, CEP 22460 Rio de
  Janeiro, Brasil.  \newline e-mail: \rm \texttt{echavez@imca.edu.pe} }

\address{\noindent IMPA, Estrada Dona Castorina 110, CEP 22460 Rio de
  Janeiro, Brasil and CNRS UMR 6085, Universit\'e de Rouen, Avenue de
  l'Universit\'e, BP.12, Technop\^ole du Madril\-let, F76801
  Saint-\'Etienne-du-Rouvray, France.  \newline e-mail: \rm
  \texttt{landim@impa.br} }

\begin{document}

\begin{abstract}
  Let $\bb T^d_N$, $d\ge 2$, be the discrete $d$-dimensional torus
  with $N^d$ points. Place a particle at each site of $\bb T^d_N$ and
  let them evolve as independent, nearest-neighbor, symmetric,
  continuous-time random walks. Each time two particles meet, they
  coalesce into one. Denote by $C_N$ the first time the set of
  particles is reduced to a singleton. Cox \cite{c89} proved the
  existence of a time-scale $\theta_N$ for which $C_N/\theta_N$
  converges to the sum of independent exponential random
  variables. Denote by $Z^N_t$ the total number of particles at time
  $t$. We prove that the sequence of Markov chains
  $(Z^N_{t\theta_N})_{t\ge 0}$ converges to the total number of
  partitions in Kingman's coalescent.
\end{abstract}

\maketitle

\section{Introduction}
\label{sec-1}

Fix $d\ge 2$, and denote by $\bb T^d_N = \{0, \dots, N-1\}^d$ the
discrete, $d$-dimensional torus with $N^d$ points.  Consider
independent, nearest-neighbor, symmetric, continuous-time coalescing
random walks evolving on $\bb T^d_N$. This dynamics can be informally
described as follows. Place a particle at each point of $\bb
T^d_N$. Each particle evolves, independently from the others, as a
continuous-time random walk which jumps from $x$ to $x \pm e_j$ with
probability $1/2d$, where the summation is taken modulo $N$ and
$\{e_1, \dots, e_d\}$ stands for the canonical basis of $\bb R^d$.
Whenever a particle jumps to a site occupied by another particle, the
two particles coalesce into one.

Let $C_N$ be the first time the set of particles is reduced to a
singleton, and let $s_N = N^d$ in dimension $d\ge 3$, $s_N= N^2 \log
N$ in dimension $2$. Cox \cite{c89} proved that $C_N/s_N$ converges in
distribution to a random variable $\tau$ which can be expressed as 
\begin{equation}
\label{eqf1}
\tau \;=\; \sum_{k \ge 2} T_k \;,
\end{equation}
where $(T_k)_{k\ge 2}$ is a sequence of independent, exponential random
variables whose expectations are given by
\begin{equation*}
E[T_{n}] \;=\; \frac 2{n\, (n-1)}\; ,\quad
\text{for } n\geq 2 \;.
\end{equation*}

This result directs us to Kingman's coalescent \cite{k82}, a dynamic
which describes a continuous-time Markov process on the equivalence
relations of
\begin{equation*}
\bb N \;=\; \{1, 2, \dots \}\;.
\end{equation*}
Here we focus our interest in the process $(\ms N_{t})_{t\geq 0}$
which records the number of equivalence classes in Kingman's
coalescent. The process $\ms N_{t}$ is a pure death process on
$\mathbb{N}\cup\{\infty\}$, starting at $\infty$, finite at any
positive time, and jumping from $k$ to $k-1$ at rate $k(k-1)/2$.  A
path of $(\ms N_{t})_{t\geq0}$ can be sampled as follows. Recall the
definition of the random variables $(T_n)_{n\ge 2}$, and set
$T_{1}=\infty$. Note that with probability one
$\sum_{n=2}^{\infty}T_{n}<\infty$ and so
\begin{equation*}
\Big[\sum_{n=k+1}^{\infty}T_{n},\sum_{n=k}^{\infty}T_{n}\Big),\quad k\in\mathbb{N}\;,
\end{equation*}
turns to be a partition of $(0,\infty)$. Set $\ms N_0 = \infty$ and,
for every $t>0$ and $k\geq1$, define
\begin{equation}\label{nec}
\ms N_t = k \iff \sum_{n=k+1}^{\infty} T_n \le t < \sum_{n=k}^{\infty} T_n \;.
\end{equation}
Notice that this process is not
continuous at $t=0$ unless every neighborhood of $\infty\in\mathbb{N}
\cup\{\infty\}$ has finite complement. 

We shall use an alternative description of this process, more suitable
to our purposes. Consider the bijection
\begin{equation*}
\begin{array}{rcl}
\{1,2,\dots,\infty\}&\to&S:=\{1,1/2,1/3,\dots,0\}\\
x&\mapsto&1/x\;,
\end{array}
\end{equation*}
taking $\infty$ to $0$, and endow $S$ with the standard differential
structure inherited by the real line. The first result of this article
characterizes the law of
\begin{equation}
\label{xx}
\ms X_t = 1/\ms N_t \;,\quad t\ge 0\;, \quad \textrm{(where $1/\infty=0$)}
\end{equation}
as the unique solution of a martingale problem. 

The second main result of the article asserts that in an appropriate
time-scale the process which records the [inverse of the] total number
of particles at a given time converges in the Skorohod topology to
$\ms X_t$. \smallskip

Since Cox' article \cite{c89}, the asymptotic behavior of the
coalescence time $C_N$ has been the subject of several papers.
Consider a connected graph $\bb G_N$ with $N$ vertices. If $\bb G_N$
is the complete graph, the distribution of $C_N$ can be computed
exactly and the process which records the total number of particles is
Markovian. This example is called the mean-field model, and one
expects that, under some mixing conditions on the random walk on the
graph $\bb G_N$, the asymptotic behavior of the coalescence time $C_N$
resembles the one of the mean field model.

Denote by $h_N$ the expected hitting time of a vertex starting from
the stationary distribution, and by $t_N$ the expected meeting time of
two independent random walks over $G_N$, both starting from the
stationarity state. Aldous and Fill \cite[Chapter 14]{af01}
conjectured in Open Problem 12 that under some mixing conditions
$E[C_N]$ is of the same order of $h_N$, as in the mean-field case.

Durrett \cite{dur10} proved mean field behavior in a small world
random graph and Cooper, Frieze and Radzik \cite[Theorem 8]{cfr09} in
random $d$-regular graphs.  Oliveira \cite{Oli12, Oli13} showed that
under some reasonable mixing conditions $C_N/t_N$ converges to $\tau$,
the random time introduced in \eqref{eqf1}, in transitive, reversible,
irreducible Markov chains. \smallskip

Our motivation to consider this problem comes from the theory of
metastable Markov chains. We proposed in \cite{bl2, bl7} a general
method, based on the characterization of Markov processes as solutions
of martingale problems, to show that projections of Markov chains on
smaller state spaces are asymptotically Markovian. Coalescing random
walks fit perfectly in this framework, as it is expected that the total
number of particles evolves asymptotically as Kingman's coalescent.

This article leaves some open questions. It would certainly be
interesting to extend the results presented here to the random graphs
covered by Oliveira \cite{Oli13} or to non-reversible dynamics, but
also to consider the dynamics which keeps track of the total number of
particles which coalesced with each particle present at a given time.
This later dynamics is related to a Wright-Fisher diffusion, already
examined by Cox \cite{c89} and Chen, Choi and Cox \cite{ccc16}.

\section{Notation and Results}
\label{sec-1}

Denote by $p$ the probability measure on $\bb Z^d$ given by 
\begin{equation}
\label{22}
p(x) \;=\; \frac 1{2d} \text{ if $x\in \{\pm e_1, \dots, \pm e_d\}$}\;,
\quad \text{and $p(x)= 0$\, otherwise}\;.
\end{equation}
Let $E_N$ be the family of nonempty subsets of $\bb T^d_N$. The
coalescing random walks introduced in the previous section is the
$E_N$-valued, continuous-time Markov chain, represented by $\{A_N(t) :
t\ge 0\}$, whose generator $L_N$ is given by
\begin{equation}\label{ln}
(L_N f)(A) \;=\; \sum_{x\in A} \sum_{y\not \in A} p(y-x) \{ f(A_{x,y}) -
f(A) \} \;+\; \sum_{x\in A} \sum_{y\in A} p(y-x) \{ f(A_{x}) -
f(A) \}\;,
\end{equation}
where $A_{x,y}$ (resp. $A_x$) is the set obtained from $A$ by
replacing the point $x$ by $y$ (resp. removing the element $x$):
\begin{equation*}
A_{x,y} \;=\; [A \setminus \{x\}] \cup \{y\} \;, \quad
A_{x} \;=\; A \setminus \{x\} \;.
\end{equation*}

\subsection{Kingman's coalescent}
Recall from the previous section the definition of Kingman's
coalescent $(\ms N_{t})_{t\geq0}$ and the definition of the set
$S$. Denote by $D(\mathbb{R}_{+},S)$ the space of $S$-valued,
right-continuous trajectories with left-limits, endowed with the
Skorohod topology.  The respective coordinate maps are denoted by
\begin{equation*}
X_{t}:D(\mathbb{R}_{+},S)\to S\;,\quad t\geq0\;.
\end{equation*}
Consider the canonical filtration
\begin{equation*}
\mc G_{t}:=\sigma(X_{s}:0\leq s\leq t)\;,\quad t\geq0\;.
\end{equation*}
It is known that $\mc G_{\infty}:=\sigma(X_{t}:t\geq0)$ coincides with
the corresponding Borel $\sigma$-field on $D(\mathbb{R}_{+},S)$. Let
$C^{1}(S)$ be the set of functions $f:S\to\mathbb{R}$ of class
$C^{1}$, that is $f\in C^{1}(S)$ is the restriction to $S$ of a
continuously differentiable function defined on a neighborhood of
$S$. For each $f\in C^{1}(S)$ define $\ms Lf:S\to\mathbb{R}$ as
\begin{equation}
\label{def1}
(\ms Lf)(y):=
\begin{cases}
\binom{n}{2}\, \Big\{ f\Big(\frac 1{n-1}\Big) - f\Big(\frac 1n\Big) \Big\}\;,
&\quad\textrm{if } y=\frac{1}{n}\;\textrm{and } n\ge 2\;,\\
0\;,&\quad\textrm{if }y=1\;,\\
(1/2)f'(0)\;,&\quad\textrm{if }y=0\;.
\end{cases}
\end{equation}
The following proposition guarantees existence and uniqueness for the
$\big(C^{1}(S),\ms L\big)$-martingale problem and that $(\ms
X_{t})_{t\geq0}$, defined in \eqref{xx}, provides the unique solution
starting at $0\in S$.

\begin{proposition}
\label{prop1}
For each $x\in S$, there exists a unique solution for the
$\big(C^{1}(S),\ms L\big)$-martingale problem starting at $x$. That
is, there exists a unique probability measure $\ms P_x$ on the
measurable space $\big(D(\bb R_+,S), \mc G_\infty \big)$ such that
$\ms P^{}_x[X_0=x]=1$ and, for every $f\in C^{1}(S)$,
\begin{equation}
\label{24}
f(X_t) - \int_0^t (\ms L f)(X_s)\,ds \;,\quad t\ge 0\;,
\end{equation}
is a $\ms P_x$-martingale with respect to $(\mc
G_t)_{t\geq0}$.
Moreover, $\ms P_0$ coincides with the law of $(\ms X_t)_{t\ge 0}$.
\end{proposition}

\subsection{Main result} 
Recall that $E_{N}$ stands for the set of nonempty subsets of
$\bb T_{N}^{d}$.  Consider the partition of $E_{N}$ according to the
number of elements:
\begin{equation}
\label{parti}
E_{N}=\bigcup_{n\in\mathbb{N}}\ms E_{N}^{n}\;,\quad\text{where}
\quad \ms E_{N}^{n}:=\{A\subset \bb T^d_N :|A|=n\}\,,\;\; n\in\bb N\;, 
\end{equation}
and $|A|$ stands for the number of elements of $A$. Let
$\Psi_N:E_N \to S$ be the projection corresponding to partition
\eqref{parti}
\begin{equation*}
\Psi_N(A) = 1/|A|\;,\quad A\in E_{N}\;.
\end{equation*}

For each $A\in E_{N},$ let $\mb P^N_A$ denote a probability measure
under which the process $\big(A_{N}(t)\big)_{t\geq0}$ corresponds to a
coalescing random walk on $\bb T_{N}^{d}$ starting at $A$, i.e. a
Markov chain with state space $E_{N}$ and generator $L_{N}$ (defined
in \eqref{ln}) such that $\mb P_{A}^{N}[A_{N}(0)=A]=1$. When
$A=\bb T^d_N$, we denote $\mb P^N_A$ simply by $\mb P^N$. Expectation
with respect to $\mb P^N_A$, $\mb P^N$ is represented by $\mb E^N_A$,
$\mb E^N$, respectively. 

Consider two independent random walks $(x_{t}^{N})_{t\geq0}$ and
$(y_{t}^{N})_{t\geq0}$ on $\bb T_{N}^{d}$, both with jump probability
given by $p(\cdot)$, starting at the uniform distribution. Let
$\theta_{N}$ be the expected meeting time:
\begin{equation*}
\theta_{N}\;:=\; E\big[\,\min\{t\geq0 : x_{t}^{N}=y_{t}^{N}\}\,\big]\,.
\end{equation*}
Since $x_{t}^{N}-y_{t}^{N}$ evolves as a random walk speeded-up by
$2$, $\theta_{N}$ represents the expectation of the hitting time of
the origin for a simple symmetric random walk speeded-up by $2$ which
starts from the stationary state. By \cite[Proposition 6.10]{bl2}, we
may express this expectation in terms of capacities. Sharp bounds for
the capacity then provide an asymptotic formula for $\theta_N$.  

Consider a continuous-time, random walk $(x_t)_{t\ge 0}$ on $\bb Z^d$
with jump probabilities given by \eqref{22} and which starts from the
origin. Assume that $d\ge 3$, and denote by $\tau_1$ the time of the
first jump, $\tau_1 = \inf\{t\ge 0 : x_t \not = 0\}$, and by $H^{+}$
the return time to the origin:
$H^{+} = \inf\{t\ge \tau_1 : x_t = 0\}$. Let $v_d$ be the escape
probability: $v_d = P[H^{+} = \infty]$. By the argument presented in
the previous paragraph, by \cite[Corollary 6.8]{jlt1} in dimension
$d\ge 3$, and by \cite[Corollary 6.12]{jlt1} in dimension $2$,
\begin{equation}
\label{thetan}
\begin{aligned}
\lim_{N\to \infty} \frac{\theta_{N}}{N^{d}}
\;=\;\frac{1}{2\, v_d}\, \quad \text{in dimension } d\ge 3\;, \\  
\lim_{N\to \infty} \frac{\theta_{N}}{N^{2}\log N} \;=\; 
\frac 1{\pi}  \quad \text{in dimension } d=2\;.
\end{aligned}
\end{equation}
The factor $2$ in the denominator appears because the process has been
speeded-up by $2$.  In particular, in $d=2$, $1/\pi$ should be
understood as $(1/2) (2/\pi)$.

Consider the rescaled reduced process
\begin{equation}\label{redu}
\bb X_N(t) \;=\; \Psi_N(A_N(\theta_N t)) \;, \quad t\ge 0 \;.
\end{equation}
Notice that $\bb X_N(t)$ is not a Markov chain, but only a hidden
Markov chain.  Denote by $\ms P^N$ the probability law on $\big(D(\bb
R_+, S),\mc G_{\infty}\big)$ induced by the reduced process $\big(\bb
X_N(t)\big)_{t\geq0}$ under $\mb P^{N}$ (i.e. starting from all
vertices in $\bb T^d_N$ occupied). The main result of this article
reads as follows

\begin{theorem}
\label{th1}
For every $d\ge 2$, the sequence of measures $\ms P^N$ converges to $\ms P_0$.
\end{theorem}

It follows from Theorem \ref{th1} that, under $\mb P^{N}$,
\begin{equation*}
\big(\bb X_{N}(t)\big)_{t\geq0}\;\xrightarrow[]
{\textrm{Law}}\;(\ms X_{t})_{t\geq0}\;,\quad\text{for } d\geq2.
\end{equation*}
The scaling limit for the coalescing times obtained in \cite{c89}
immediately follows from these results.

\begin{remark}
\label{rm1}
The proofs apply to the case in which the jump probability $p(\cdot)$
is symmetric and has finite range. It also applies if the initial
condition $\bb T^d_N$ is replaced by a finite set $A=\{x_1, \dots,
x_n\}$ whose points are scattered: $\Vert x_i- x_j\Vert \ge a_N$ for
$1\le i\not = j\le n$, where $a_N$ is the sequence introduced in
\eqref{02}. 
\end{remark}

\subsection{ Sketch of the proof}\label{sket}
The proof is divided in two steps. We first show that the sequence
$(\ms P^N)$ is tight, and then we guarantee uniqueness of limit points
by proving that every limit point solves the
$\big(C^{1}(S), \ms L\big)$-martingale problem.

For the later step, consider a smooth function $f:\bb R\to\bb R$, and
denote by $M_N(t)$ the martingale given by
\begin{equation*}
f(\bb X_N(t)) \,-\, f(\bb X_N(0)) \,-\,
\int_0^t \theta_N \, (L_N f)\big(\Psi_N(A_N(s\theta_N))\big)\, ds\;.
\end{equation*}
Since
\begin{equation*}
(L_N f)(\Psi_N(A)) \;=\; R(A) \, \Big\{ f \Big(\frac x{1-x} \Big) - f(x) \Big\}\;, 
\end{equation*}
where $x=\Psi_N(A)$, and $R(A)$ is the jump rate given by
\begin{equation}
\label{27}
R(A) \;=\; \sum_{x\in A} \sum_{y\in A\setminus\{x\}} p(y-x)\;,
\end{equation}
the martingale $M_N(t)$ can be written as
\begin{equation*}
f(\bb X_N(t)) \,-\, f(\bb X_N(0)) \,-\, \theta_N\, \int_0^{t} \, R(A_N(s\theta_N))\,
\Big\{ f \Big(\frac {\bb X_N(s)}{1- \bb X_N(s)} \Big) - f(\bb X_N(s)) \Big\} \, ds\;.
\end{equation*}

If the martingale $M_N(t)$ were expressed in terms of the process
$\bb X_N$, that is if $\theta_N \, R(A_N(s\theta_N)) = r(\bb X_N(s))$,
we could pass to the limit and argue that
\begin{equation}
\label{28}
f(\bb X(t)) \,-\, f(\bb X(0)) \,-\, \int_0^{t} \, r(\bb X (s))\,
\Big\{ f \Big(\frac {\bb X(s)}{1- \bb X(s)} \Big) - f(\bb X(s)) \Big\} \, ds\;.
\end{equation}
is a martingale for every limit point $\ms P^*$ of the sequence
$\ms P^N$. This result together with the uniqueness of solutions of
the martingale problem \eqref{28} on
$\big(D(\bb R_{+},S), \mc G_{\infty}\big)$ would yield the uniqueness
of limit points.

The previous argument evidences that the main point of the proof
consists in ``closing'' the martingale $M_N(t)$ in terms of the
reduced process $\bb X_N(s)$, that is, that the major difficulty lies
in the proof of the existence of a function $r: S \to \bb R$ such that
\begin{equation*}
\int_0^{t} \Big\{ \theta_N\,  R(A_N(s\theta_N)) \,-\, r(\bb X_N(s)) \Big\} \,
g (\bb X_N(s)) \, ds\;\longrightarrow \; 0
\end{equation*}
for all smooth functions $g: \bb R \to \bb R$. This is the so-called
``replacement lemma'' or the ``local ergodic theorem''. One has to
replace a function $\theta_N R(A)$ which does not vanish only in a
tiny portion of the state space (in the present context for subsets of
$(\bb T^d_N)^n$ which contain at least two neighborhing points) and
which is very large (here of order $\theta_N)$ when it does not
vanish, by a function of order $1$ in the entire space.

The statement of the local ergodic theorem requires some notation. 
Denote by $D(\bb R_+, E_N)$ the right-continuous trajectories $\omega
: \bb R_+ \to E_N$ wich have left-limits. Let
\begin{equation}
\label{14}
\bs r \Big( \frac 1n \Big) \;:=\;
\bs \lambda(n) \;:=\; \binom{n}{2}\;, \quad n\ge 2\;.
\end{equation}

\begin{proposition}
\label{p01}
Let $F: \bb N \to \bb R$ be a function which eventually vanishes:
there exists $k_0\ge 1$ such that $F(k)=0$ for all $k\ge k_0$. Let
$t_{0}>0$ and let $(B^{N}:D(\bb R_{+},E_{N})\to\mathbb{R};N\geq1)$ be
a sequence of uniformly bounded functions, with each $B^{N}$
measurable with respect to
$\sigma\big(A_{N}(s\theta_{N}):0\leq s\leq t_{0}\big)$. Then, for
every $t>t_{0}$,
\begin{equation*}
\lim_{N\to\infty} \mb E^N \Big[B^{N}\,\int_{t_{0}}^{t}
\big \{\theta_{N} R(A_N(s\theta_{N}))
- \mf n_{s\theta_{N}} \big\} \,  F(|A_N(s\theta_{N})|) \, ds\,\Big]\;=\; 0\;, 
\end{equation*}
where $\mf n_s = \bs \lambda (|A_N(s)|)$.
\end{proposition}

The article is organized as follows. In Section \ref{sec2}, we present
the results on coalescing random walks needed in the proof of
Proposition \ref{p01}, which is presented in the following section.
In Section \ref{teo12}, we prove Theorem \ref{th1} and, in Section
\ref{sec5}, Proposition \ref{prop1}.

\section{Coalescing random walks on $\bb T^d_N$}
\label{sec2}

We present in this section some results on coalescing randoms walks
obtained by Cox \cite{c89}: Propositions \ref{p02}, \ref{lem3} and
\ref{lem4}. We start with some notation.  

Throughout this section, $P^N_x$ represents the distribution of a
$\bb T^d_N$-valued random walk, \emph{speeded-up by $2$}, whose jump
probability is $p(\cdot)$, introduced in \eqref{22}, and initial
position is $x$. Denote by $p_t(x,y) = P^N_x[x(t)=y]$ the transition
probabilities of this process and by $\pi_N$ its stationary state, which is
the uniform measure on $\bb T^d_N$.

The first result, Proposition (4.1) in \cite{c89}, provides a bound on
the expectation of the number of particles still present at time
$t$. Let
\begin{equation*}
g_N(t) \;=\; 
\begin{cases}
N^2 \, t^{-1} \log (1+t) & d=2\;, \\
N^d/t & d\ge 3\;.
\end{cases}
\end{equation*}

\begin{proposition}
\label{p02}
There exists a finite constant $c_d$ such that
\begin{equation*}
\mb E^N [ \, |A_N(t)|\, ] \;\le\; c_d\, \max\{1,  g_N(t)\} 
\end{equation*}
for all $t> 0$, $N\ge 1$.
\end{proposition}

Recall from \eqref{parti} that we denote by $\ms E_{N}^{n}$ the
subsets of $\bb T^d_N$ with $n$ elements.  Denote by $\tau_j$, $j\ge 1$,
the time when the process $A_N(t)$ is reduced to a set of $j$
elements:
\begin{equation}
\label{23}
\tau_j \;=\; \inf \big\{ t\ge 0 : |A_N(t)| = j \big\}
\;=\; \inf \big\{ t\ge 0 : A_N(t) \in \ms E_{N}^{j} \big\}\;.
\end{equation}

\begin{lemma}
\label{a04b}
There exists a finite constant $C_0$ such that for all $j\ge2$,
\begin{equation*}
\max_{A \in \ms E_{N}^{j}} 
\frac 1{\theta_N}\, \mb E^N_{A} \big[ \tau_{j-1} \big] 
\;\le\; C_0\;.
\end{equation*}
\end{lemma}

\begin{proof}
Fix two points $x$, $y$ in $A$ and denote by $\tau_{x,y}$ the first
time these particles meet: $\tau_{x,y} = \inf \{ t>0 : x(t) =
y(t)\}$. Since $\tau_{j-1} \le \tau_{x,y}$, and since the difference
$x(t)-y(t)$ evolves as a random walk speeded-up by $2$, the
expectation appearing in the statement of the lemma is bounded by
\begin{equation*}
\max_{x\in \bb T_N} \frac 1{\theta_N}\,  E^N_{x} \big[ H_0 \big]\;,
\end{equation*}
where $H_0$ represents the hitting time of the origin. By
\cite[Proposition 10.13]{lpw09}, this quantity is bounded by a finite
constant independent of $N$.
\end{proof}

It follows from the previous result that for every $j\ge2$,
\begin{equation}
\label{a04}
\lim_{M\to\infty} \limsup_{N\to\infty} \max_{A \in \ms E_{N}^{j}} 
\mb P^N_{A} \big[ \tau_{j-1} \,\ge\, M \,\theta_N \big] \;=\; 0\;.
\end{equation}

Denote by $\Vert \mu - \nu \Vert_{\rm TV}$ the total variation
distance between two probability measures, $\mu$, $\nu$, defined on a
countable state space $E$:
\begin{equation*}
\Vert \mu - \nu \Vert_{\rm TV} \;=\; \frac 12 \, \sum_{a\in E} |\, \mu(a)
- \nu(a)\,|\;.
\end{equation*}

Hereafter, the symbol $\alpha_N \ll \beta_N$, for two non-decreasing
sequences $\alpha_N$, $\beta_N$, means that $\alpha_N/\beta_N \to
0$. Denote by $a_N$ an increasing sequence such that $1\ll a_N \ll
N$. In dimension $2$, assume further that $N/\sqrt{\log N} \ll
a_N$. Denote by $\mf G_N(n,a_N)$ the scattered subsets of $E_N$. These
are the sets $A=\{y_1, \dots, y_n\}$ in $\ms E_{N}^{n}$ such that
\begin{equation}
\label{02}
\min_{i\not = j} |y_i - y_j| \;\ge\; a_N \;.
\end{equation}

\begin{lemma}
\label{lem16}
For every $n \ge 2$, $t>0$,
\begin{equation*}
\lim_{N\to \infty} \max_{A \in \ms E_{N}^{n}} \mb P^N_A \Big[ A_N(t\theta_N)
\not\in \ms E_{N}^{1} \, \cup \, \bigcup_{k=2}^n \mf G_N(k, a_N) \Big] \;=\; 0\;. 
\end{equation*}
\end{lemma}
 
\begin{proof}
Since $n$ is finite and since the difference of two random walks evolves as
a random walk speeded-up by $2$, this assertion follows from the claim
that for  every $t>0$ 
\begin{align*}
& \lim_{N\to \infty} \max_{x\in \bb T^d_N} \mb P^N_{\{0,x\}} \big[ A_N(t\theta_N)
\not\in \ms E_{N}^{1} \cup \mf G_N(2,a_N) \big] \\
&\quad \;=\;
\lim_{N\to \infty} \max_{x\in \bb T^d_N} 
P^N_{x} \big[ H_0 > t\theta_N \,,\, |x(t\theta_N)| \le a_N  \big] \;=\; 0\;. 
\end{align*}
By the Markov property, the previous probability is bounded by
\begin{equation*}
E^N_{x} \Big[ 
P^N_{x(t\theta_N/2)} \big[\, |x(t\theta_N/2)| \le a_N  \big]\, \Big]\;.
\end{equation*}
Recall from the beginning of this section that $\pi_N$ represents the
stationary state of the random walk on $\bb T^d_N$. The previous
expectation is less than or equal to
\begin{equation*}
P^N_{\pi_{N}} \big[ |x(t\theta_N/2)| \le a_N  \big] \;+\;2 \,
\Vert \pi_{N}(\cdot) - p_{t\theta_{N}/2}(x,\cdot) \Vert_{\rm TV}\;,
\end{equation*}
where $p_{t}(x,y)$ represents the transition probabilities of a random
walk evolving on $\bb T^d_N$ speeded-up by $2$. The first term is
bounded by $C_0 (a_N/N)^d\to 0$, while the second one vanishes because
$\theta_N \gg t^N_{\rm mix}$.
\end{proof}

\begin{corollary}
\label{obsaN} 
For every $t>0$,
\begin{equation*}
\lim_{N\to\infty}\mb P^{N}\Big[ A_N(t\theta_N)
\not\in \ms E_{N}^{1} \, \cup \, \bigcup_{k=2}^{N^{d}} \mf G_N(k, a_N)
\Big] \;=\; 0\;.
\end{equation*}
\end{corollary}

\begin{proof}
Fix $t>0$, and let $\ms H_{s}=\big\{A_N(s\theta_N)\not\in \ms
E_{N}^{1} \, \cup \, \bigcup_{k=2}^{N^{d}} \mf G_N(k, a_N) \big\}$,
$s>0$. Clearly, for every $M>0$,
\begin{equation*}
\mb P^{N}[\ms H_{t}]\,\leq\,
\mb P^{N}\big[\, |A_{N}(t\theta_{N}/2)|\leq M\,,\,\ms H_{t} \, \big]
\;+\; \mb P^{N}\big[\, |A_{N}(t\theta_{N}/2)|> M\, \big] \;.
\end{equation*}
By Proposition \ref{p02}, the second term is bounded by $C(d,t)/M$,
where $C(d,t)$ is a constant depending only on $d$ and $t$. Hence,
by the Markov property,
\begin{equation*}
\mb P^{N}[\ms H_{t}]\;\leq\; \max_{2\le k\le M}
\max_{A\in\ms E_{N}^{k}}\mb P_{A}^{N}[\ms H_{t/2}]\;+\;\frac{C(d,t)}{M}\;.
\end{equation*}
By Lemma \ref{lem16}, the first term on the right-hand side vanishes
as $N\to\infty$ for every $M\ge 2$. This proves the corollary.
\end{proof}

\begin{proposition}
\label{lem3}
For every $2\le j<k$,
\begin{equation*}
\lim_{N\to\infty} \max_{A\in \mf G_N(k,a_N)} 
\mb P^N_{A} \big[ A_N(\tau_{j}) 
\not\in \mf G_N(j,a_N) \big] \;=\; 0\;.
\end{equation*}
\end{proposition}

\begin{proof}
Fix $2\le j<k$.  By \eqref{a04}, it is enough to prove that for
all $M>0$,
\begin{equation*}
\lim_{N\to\infty} \max_{A\in \mf G_N(k,a_N)} 
\mb P^N_{A} \big[ A_N(\tau_{j}) \not\in \mf G_N(j,a_N) \,,\,
\tau_{j} \le M\, \theta_N \big] \;=\; 0\;.
\end{equation*}
This is exactly assertions (3.7) and (3.8) in \cite{c89}.
\end{proof}

Denote by $\pi_{N}^{n}$, $n\ge 2$, the uniform measure on $\ms E^n_N$.
Recall the definition of $\bs \lambda(\cdot)$ given in \eqref{14}.
Next proposition is a weak version of \cite[Theorem 5]{c89}.

\begin{proposition}
\label{lem4}
For all $j\ge 2$,
\begin{equation*}
\lim_{N\to\infty} 
\mb P^N_{\pi_{N}^{j}} \big[\, \tau_{j-1} \ge t\, \theta_N \, \big] 
\;=\; e^{ - \bs \lambda (j)\, t} \;.
\end{equation*}
\end{proposition}

It follows from the previous result that for every $n\ge 1$,
\begin{equation}
\label{04}
\lim_{\delta \to 0} \limsup_{N\to \infty} \mb P^N \big[ \tau_n \le
\delta \theta_N \big] \;=\; 0\;.
\end{equation}
Indeed, fix $n\ge 1$ and consider a set $A\in \ms E_{N}^{n+1}$.  Since
$A\subset \bb T^d_N$,
$\mb P^N [ \tau_n \le \delta \theta_N ] \le \mb P^N_A [ \tau_n \le
\delta \theta_N]$.
Averaging over $A$ with respect to $\pi_{N}^{n+1}$ we obtain
that
$\mb P^N [ \tau_n \le \delta \theta_N ] \le \mb
P^N_{\pi_{N}^{n+1}} [ \tau_n \le \delta \theta_N]$.
By Proposition \ref{lem4}, the previous quantity vanishes as
$N\to\infty$ and then $\delta\to 0$.

Denote by $\gamma_N$ a sequence much larger than the mixing time and
much smaller than the hitting time:
\begin{equation}
\label{09}
t^N_{\text{mix}} \;\ll\; \gamma_N \;\ll\; \theta_N\;.
\end{equation}
Let $(\ell_N : n\ge 1)$ be a sequence such that $1\ll \ell_N \ll N$.
In dimension 2, we assume that $N^\alpha \ll \ell_N \ll N$ for all
$0<\alpha<1$, so that
\begin{equation}
\label{13}
\lim_{N\to\infty} \frac{\ell_N}N \;=\; 0\;, \quad \lim_{N\to\infty}
\frac{\log \ell_N}{\log N} \;=\; 1\;.
\end{equation}
Note that in dimension $2$ the conditions imposed on $\ell_N$ are
weaker than the ones assumed on $a_N$ in \cite[Theorem 4]{c89}.

\begin{lemma}
\label{lem8}
For every $n\ge 2$, 
\begin{equation*}
\lim_{N\to\infty} \max_{A\in \mf G_N (n,\ell_N)}  
\mb P^N_A \big[ \tau_{n-1} \le \gamma_N \big] \;=\; 0 \;.
\end{equation*}
\end{lemma}

\begin{proof}
The probability is bounded by
\begin{equation*}
\binom n2\, \max_{\Vert x\Vert \ge \ell_N} P^N_x [H_0 \le
\gamma_N]\; ,
\end{equation*}
where, recall, $H_0$ stands for the hitting time of the origin. Since
$\gamma_N \ll \theta_N$, by equation (6.18) in \cite{jlt1}, this
expression vanishes in the limit.
\end{proof}

In the next lemma we compare the dynamics $A_N(t)$ with the one of
independent random walks.  Fix $n\ge 2$, and denote by $({\bf
  x}_{N}^{n}(t))_{t\ge 0}$, the evolution of $n$ independent random
walks on $\bb T^d_N$ with jump probabilities $p(\cdot)$ given by
\eqref{22}.  The stationary state of this dynamics, denoted by
$\pi_{N}^{\otimes n}$, is the product measure on $[\bb T^d_N]^n$ in
which each component is the measure $\pi_{N}$. 

Denote by $p^{(n)}_{t}({\bf x}, {\bf y})$ the transition probabilities
of ${\bf x}^n_N(t)$, and by $t^{N,n}_{\text{mix}}$ the corresponding
mixing time. Since the dynamics amounts to the evolution of a random
walk on $\bb T^{nd}_N$, there exist constants $0<c(d,n) < C(d,n)
<\infty$ such that $c(d,n)N^2 \le t^{N,n}_{\text{mix}} \le C(d,n) N^2$
(cf. \cite[Section 5.3 and 7.4]{lpw09}).

Denote by $x_{j}(t) \in \bb T^d_N$ the $j$-th coordinate of ${\bf
  x}_{N}^{n}(t)$, $1\le j\le n$.  Up to time $\tau_{n-1}$ the process
$A_N(t)$ evolves as $\{{\bf x}_{N}^{n}(t)\} :=\{x_1(t), \dots,
x_n(t)\}$. More precisely, fix $A = \{a_1, \dots, a_n\} \in \ms
E^n_N$, and let
\begin{equation*}
\ms E_{N}^{\le n} \;:=\; \bigcup_{k=1}^n \ms E_{N}^{k}\;.
\end{equation*}
There exists a probability measure on $D\big(\bb R_{+}, \ms
E^{\le n}_N \times (\bb T^{d}_{N})^{n}\big)$, denoted by
$\widehat{\mb P}^{N}_{A}$, which fulfills the following conditions. The
distribution of the first, resp. second, coordinate corresponds to the
distribution induced by $A_N(t)$, resp.  ${\bf
  x}^n_N(t)$. Furthermore, $A_N(0)=A$, ${\bf x}^n_N(0) = (a_1, \dots,
a_n)$, and $A_N(t)=\{{\bf x}_{N}^{n}(t)\}$ for all $0\le t\le
\tau_{n-1}$, $\widehat{\mb P}^{N}_{A}$ almost surely.

\begin{lemma}
\label{lf1}
Fix $n\ge 2$. Let $F_N: \ms E^{\le n}_N \to \bb R$ be a sequence of
uniformly bounded functions, $\Vert F\Vert := \sup_{N\ge 1} \max_{A\in
  \ms E^{\le n}_N} |F_N(A)| < \infty$, and let $(\beta_N)_{N\ge 1}$
be a non-negative sequence. Then, for every $A=\{a_1, \dots a_n\}\in
\ms E^n_N$,
\begin{align*}
& {\mb E}^N_A \Big[ F_N(A_N(\beta_N)) 
\, \mf 1\{ \tau_{n-1} > \beta_N \} \, \Big]\;-\;
E_{\pi^n_N} \big[ F_N \big] \\
& \quad =\; -\, \widehat{\mb E}^N_A \Big[ F_N(\{{\bf x}_{N}^{n}(\beta_N)\}) 
\, \mf 1\big\{ \tau_{n-1} \le \beta_N \,,\, 
\{{\bf x}_{N}^{n}(\beta_N)\} \in \ms E^n_N\, \big\} \, \Big]
\;+\; R_{N} \;,
\end{align*}
where 
\begin{equation*}
\big| R_{N} \big| \;\le\; \Vert F\Vert\,  \Big\{ 2\, 
\Vert p^{(n)}_{\beta_N}({\bf a}, \cdot) - \pi^{\otimes n}_N(\cdot)
\Vert_{\rm TV}  \;+\; c_N \big\}
\end{equation*}
and $\lim_{N\to \infty} c_N =0$.
\end{lemma}

\begin{proof}
Fix $A=\{a_1, \dots a_n\}\in\ms E^n_N$.
We may rewrite the expectation appearing in the statement of the lemma
as 
\begin{equation*}
\widehat{\mb E}^N_A \Big[ F_N(A_N(\beta_N)) 
\, \mf 1\{ \tau_{n-1} > \beta_N \} \, \Big]\;.
\end{equation*}
Since $A_N(t) = \{{\bf x}_{N}^{n}(t)\}$ in the time interval $[0,
\tau_{n-1}]$, we may replace in the previous equation $A_N(\beta_N)$
by $\{{\bf x}_{N}^{n}(\beta_N)\}$ and then add the indicator function
of the set $\{{\bf x}_{N}^{n}(\beta_N)\} \in \ms E^n_N$. After these
replacements, the previous expression becomes
\begin{align*}
& \widehat{\mb E}^N_A \Big[ F_N(\{{\bf x}_{N}^{n}(\beta_N)\}) 
\, \mf 1\{\, \{{\bf x}_{N}^{n}(\beta_N)\} \in \ms E^n_N  \, \} \Big] \\
&\quad \;-\; \widehat{\mb E}^N_A \Big[ F_N(\{{\bf x}_{N}^{n}(\beta_N)\}) 
\, \mf 1\big\{ \tau_{n-1} \le \beta_N \,,\, 
\{{\bf x}_{N}^{n}(\beta_N)\} \in \ms E^n_N\, \big\} \, \Big] \;.
\end{align*} 

We estimate the first term. Recall that we denote by $p^{(n)}_{t}({\bf
  x}, {\bf y})$ the transition probabilities of ${\bf x}^n_N(t)$. With
this notation, we may write this term as
\begin{equation*}
\sum_{{\bf x} \in [\bb T^d_N]^n} F_N(\{{\bf x}\}) 
\, \mf 1\{\, \{{\bf x}\} \in \ms E^n_N\, \}  
\; \pi^{\otimes n}_N({\bf x}) \;+\; R^{(1)}_N\;,
\end{equation*}
where
\begin{equation*}
\big\vert\, R^{(1)}_N \, \big\vert  \;\le\;
2\, \Vert F\Vert\,
\Vert p^{(n)}_{\beta_N}({\bf a}, \cdot) - \pi^{\otimes n}_N(\cdot)
\Vert_{\rm TV} \;.
\end{equation*}
and ${\bf a} = (a_1, \dots, a_n)$.

To bound the first term of the penultimate formula, recall that we
denote by $\pi^n_N$ the uniform measure on $\ms E^n_N$. Let
\begin{equation}
\label{feq3}
R^{(2)}_{N,n} \;:=\; 
\sum_{A\in \ms E^n_N} \Big|\, \pi^n_N(A) \;-\; \sum_{{\bf x} \in [\bb T^d_N]^n} 
 \mf 1\{ \, \{{\bf x}\} =A \}  \; \pi^{\otimes n}_N(\bf x)\, \Big| \;.
\end{equation}
An elementary computation shows that $\lim_{N\to\infty} R^{(2)}_{N,n}
=0$ for every $n\ge 2$.  The assertion of the lemma follows from the
previous estimates.
\end{proof}

The next lemma is a consequence of \cite[Theorem 5]{c89} in dimension
$d\ge 3$. In dimension $2$ is a slight generalization since our
assumptions on $\ell_N$ are weaker.

\begin{lemma}
\label{lem14}
Let $\ell_N$ be a sequence satisfying the conditions introduced above
\eqref{13}. Then, for all $t>0$,
\begin{equation*}
\lim_{N\to\infty}
\max_{A\in \mf G(n, \ell_N)}  \Big| \, \mb P^N_A \big[ \tau_{n-1}
\ge t \theta_N \big] \,-\, e^{- \bs \lambda(n) \, t} \,\Big| \;=\; 0\;.
\end{equation*}
\end{lemma}

\begin{proof}
We present the proof in dimension $d=2$. The one in higher dimension
is analogous.  Fix a set $A = \{a_1, \dots, a_n\}$ in
$\mf G_N(n, \ell_N)$ and a sequence $1\ll t_N \ll \log N$. Recall from
the previous lemma the definition of the measure
$\widehat{\mb P}^{N}_{A}$. Since the first coordinate evolves as
$A_N(t)$,
\begin{equation*}
\mb P^N_A \big[ \tau_{n-1} \ge t \theta_N \big] \;=\;
\widehat{\mb P}^{N}_{A} \big[ \tau_{n-1} \ge t \theta_N \big]\;.
\end{equation*}
By the Markov property,
\begin{equation*}
\widehat{\mb P}^N_A \big[ \tau_{n-1} \ge t \theta_N \big] \;=\;
\widehat{\mb E}^N_A \Big[ \widehat{\mb P}^N_{A_N(\gamma_N)} \big[ \tau_{n-1} \ge t \theta_N
- \gamma_N \big] \, \mf 1\{ \tau_{n-1} > \gamma_N \} \, \Big]\;,
\end{equation*}
where $\gamma_N=t_N N^2$. 

We apply Lemma \ref{lf1} with $\beta_N=\gamma_N$ to estimate the
right-hand side. Let $F_N: \ms E^{\le n}_N \to \bb R$ be the function
defined by
\begin{equation*}
F_N(A) \;=\; {\mb P}^N_{A} \big[ \tau_{n-1} \ge t \theta_N - \gamma_N
\big] \;, \quad A\,\in\, \ms E^n_N\;,
\end{equation*}
and $F_N(A) = 0$ for $A\,\not \in\, \ms E^n_N$. By Lemma \ref{lf1},
the right hand side of the penultimate formula is equal to
\begin{equation*}
{\mb P}^N_{\pi^n_N} \big[ \tau_{n-1} \ge t \theta_N - \gamma_N \big]
\;+\; \bs R_N\;,  
\end{equation*}
where
\begin{equation*}
\big| \bs R_N \big| \;\le\;  
\widehat{\mb P}^N_A \big[ \tau_{n-1} \le \gamma_N \big]
\;+\; 2\, \Vert p^{(n)}_{\gamma_N}({\bf a}, \cdot) - \pi^{\otimes n}_N(\cdot)
\Vert_{\rm TV}  \;+\; c_N \;,
\end{equation*}
with $\lim_{N\to\infty} c_N =0$.

Each term of the previous expression is negligible.  In the first one,
we may replace $\widehat{\mb P}^N_A$ by $\mb P^N_A$, and apply Lemma
\ref{lem8} to conclude that this expression vanishes as
$N\to\infty$. The second one also vanishes in the limit because
$\gamma_N\gg t^{N}_{\rm mix}$ and $t^{N,n}_{\rm mix}$ is of the same
order of $t^{N}_{\rm mix}$.  To complete the proof of the lemma, as
$\gamma_N\ll \theta_N$, it remains to apply Proposition \ref{lem4}.
\end{proof}

Recall the properties of the sequence $a_N$ introduced in
\eqref{02}. By the previous result, for all $k>j\ge 2$,
\begin{equation}
\label{06}
\lim_{N\to\infty} \max_{A\in \mf G_N(k,a_N)} \Big|\, 
\mb P^N_A \big[\,  \tau_{j-1}-\tau_j \ge t\, \theta_N \, \big] 
\,-\, e^{ - \bs \lambda (j)\, t}  \, \Big| \;=\; 0 \;.
\end{equation}
Indeed, by Proposition \ref{lem3}, we may intersect the event appearing
inside the probability with the set $\{A_N(\tau_{j}) \in \mf
G_N(j,a_N)\}$. Then, applying the strong Markov property at time
$\tau_j$ we reduce assertion \eqref{06} to Lemma \ref{lem14}.

The next result together with the previous lemma entails the
convergence of $\mb E^N_{A_N} \big[ \tau_{n-1} / \theta_N \big]$ to
$\bs \lambda(n)^{-1}$ for any sequence $A_N\in \mf G(n, \ell_N)$.

\begin{lemma}
\label{lem11}
For every $n\ge 2$, $m\ge 1$, there exists a finite constant $C(n,m)$ such that
for all $N\ge 1$,
\begin{equation*}
\max_{A\in \ms E_{N}^{n}}  \mb E^N_A \big[ \, (\tau_{n-1}/\theta_N)^m \, \big] 
\;\le\; C(n,m) \;.
\end{equation*}
\end{lemma}

\begin{proof}
By the Markov property, for all $k\ge 1$,
\begin{equation*}
\max_{A\in \ms E_{N}^{n}} \mb P^N_A [ \tau_{n-1}/\theta_N \ge k] \;\le\;
\Big( \max_{A\in \ms E_{N}^{n}} \mb P^N_A [ \tau_{n-1}/\theta_N \ge 1] \Big)^k\;.
\end{equation*}

We claim that 
\begin{equation}
\label{11}
 \max_{A\in \ms E_{N}^{n}} \mb P^N_A [ \tau_{n-1} \ge \theta_N ] \;\le\;
\mb P^N_{\pi_{N}^{n}}  [\tau_{n-1} \ge \theta_N/2]
\;+\; \delta_N \;.
\end{equation}
where $\delta_N\to 0$.  Indeed, fix
$A=\{a_1, \dots a_n\}\in \ms E_{N}^{n}$, and apply the Markov property
to obtain that
\begin{equation*}
\mb P^N_A [ \tau_{n-1} \ge \theta_N ] \;=\;
\mb E^N_A \Big[ \mb P^N_{A_N(\theta_N/2)}  [\tau_{n-1} \ge
\theta_N/2]\, \mf 1\{ \tau_{n-1} \ge \theta_N/2 \} \, \Big]\;. 
\end{equation*}
Let $F_N: \ms E^{\le n}_N \to \bb R$ be the function
defined by
\begin{equation*}
F_N(A) \;=\; {\mb P}^N_{A} \big[ \tau_{n-1} \ge \theta_N/2
\big] \;, \quad A\,\in\, \ms E^n_N\;,
\end{equation*}
and $F_N(A) = 0$ for $A\,\not \in\, \ms E^n_N$. Since $F_N$ is
non-negative, by Lemma \ref{lf1}, the right-hand side of the
penultimate formula is bounded above by
\begin{equation*}
\mb P^N_{{\pi}_{N}^{n}}  [\tau_{n-1} \ge \theta_N/2] \;+\;  
2\,\Vert p^{(n)}_{\theta_N/2}({\bf a}, \cdot) - {\pi}_{N}^{n}(\cdot) 
\Vert_{\rm TV} \;+\; c_N \;,
\end{equation*} 
where ${\bf a} =(a_1, \dots a_n)$.  Assertion \eqref{11} follows from
the facts that $\theta_N \gg t^N_{\rm mix}$ and that
$t^{N,n}_{\rm mix}$ is of the same order of $t^{N}_{\rm mix}$.

By Proposition \ref{lem4}, under the measure $\mb P^N_{\pi_{N}^{n}}$,
$\tau_{n-1}/\theta_N$ converges weakly to an exponential random
variable of parameter $\bs \lambda(n)$. Thus, the right-hand side
of \eqref{11} converges to $e^{-\bs \lambda(n)/2}<1$. Therefore, there
exists $\delta<1$ such that for all $N\ge 1$,
\begin{equation*}
\max_{A\in \ms E_{N}^{n}} \mb P^N_A [ \tau_{n-1}/\theta_N \ge k] \;\le\;
\delta^k\;.
\end{equation*}
This proves the lemma.
\end{proof}

\begin{corollary}
\label{lem10}
For every $n\ge 2$, 
\begin{equation*}
\lim_{N\to\infty} \max_{A\in \mf G_N(n, \ell_N)}  \Big|\, 
\frac 1{\theta_N} \, \mb E^N_A [ \tau_{n-1} ] 
\, -\, \frac 1{\bs \lambda(n)}\, \Big| \;=\; 0 \;.
\end{equation*}
\end{corollary}

\begin{proof}
Fix a sequence $A_N \in \mf G_N(n, \ell_N)$, $N\ge 1$.  The
convergence in law of the sequence $\tau_{n-1}/\theta_N$ under the
measure $\mb P^N_{A_N}$ to an exponential random variable of parameter
$\bs \lambda(n)$ follows from Lemma \ref{lem14}. By the previous lemma
the sequence $\tau_{n-1}/\theta_N$ is uniformly integrable.
\end{proof}

Recall that we denote by $(e_1, \dots, e_d)$ the canonical basis of
$\bb R^d$.

\begin{lemma}
\label{lem12}
Assume that $d\ge 3$ and $n\ge 2$.  Fix a sequence of sets $A_N \in
\ms E_{N}^{n}$ such that $A_N=\{x_N,x_N\pm e_j\}\cup B_N$, where
$B_N\cup\{x_N\}$ belongs to $\mf G_N(n-1, \ell_N)$. For all $t>0$,
\begin{equation}
\label{12}
\lim_{N\to\infty} \mb P^N_{A_N} [ \tau_{n-1} \ge t \theta_N] \;=\;
v_d\, e^{- \bs \lambda(n)  t } \;.
\end{equation}
\end{lemma}

\begin{proof}
Denote by $x(t)$, $y(t)$ the position at time $t$ of the particle
initially at $x_N$, $x_N\pm e_j$, respectively. Let $D_r$, $r\ge 0$, be the
first time the distance between these particles attains $r$:
$D_r=\inf\{t>0 : \Vert x(t)-y(t)\Vert = r\}$, and let $H= D_0 \wedge
D_{\ell_N}$. As $\ell_N \ll N$, an elementary computation shows that
\begin{equation*}
\lim_{N\to\infty} \mb P^N_{A_N} [ H > N^2 ] \;=\; 0\;.
\end{equation*}
We may therefore insert the set $\{ H \le N^2\}$ in the probability
appearing in equation \eqref{12}. On the event $\{ H \le N^2\}$, when
$tN^{d-2}>1$, we have that $\{D_0<D_{\ell_N}\} \cap \{\tau_{n-1} \ge t
\theta_N\}=\varnothing$. Note that here we used that $d\geq3$. Hence,
\begin{equation*}
\mb P^N_{A_N} [ \tau_{n-1} \ge t \theta_N] \;=\;
\mb P^N_{A_N} \Big[ H \le N^2 \,,\, D_0 > D_{\ell_N} \,,\,
\tau_{n-1} \ge t \theta_N \, \Big]\;+\; o_N(1)\;,
\end{equation*}
where $o_N(1)\to 0$ as $N\to\infty$.

By the Markov property, the probability on the right hand side is
equal to
\begin{equation*}
\mb E^N_{A_N} \Big[ \mf 1\{ H \le N^2 \,,\, D_0 > D_{\ell_N} \,,\, 
\tau_{n-1} \ge N^2 \} \,
\mb P^N_{A_N(N^2)} [\tau_{n-1} \ge t \theta_N - N^2] \, \Big]\;. 
\end{equation*}
On the event $\{\tau_{n-1} \ge N^2\}$, we may replace the distribution
of $A_N(N^2)$ by the one of the position at time $N^2$ of $n$
independent random walks starting from $A_N$.  After this replacement,
we may insert in the expectation the indicator of the set $\{A_N(N^2)
\in \mf G_N(n,\ell_N)\}$ because the probability of the complement
vanishes as $N\to\infty$ [indeed, whatever the initial position of a
random walk, its probability to be a distance $\ell_N$ from the origin
at time $N^2$ vanishes]. After this insertion, we write the previous
expectation as
\begin{equation*}
e^{- \bs \lambda(n) t} \; \mb P^N_{A_N} \Big[ H \le N^2  \,,\, D_0 > D_{\ell_N}  
\,, \, A_N(N^2) \in \mf G_N(n,\ell_N) \,,\,
\tau_{n-1} \ge N^2 \Big] \;+\; R_N \;,
\end{equation*}
where the absolutely value of $R_N$ is bounded by
\begin{equation*}
\max_{A\in \mf G_N(n,\ell_N) }  \Big| \mb P^N_{A} 
[\tau_{n-1} \ge t \theta_N - N^2] \,-\, 
e^{- \bs \lambda(n)  t } \, \Big|\;.
\end{equation*}
By Lemma \ref{lem14}, this expression vanishes as $N\to\infty$. Hence,
up to this point we proved that the probability appearing in
\eqref{12} is equal to
\begin{equation*}
e^{- \bs \lambda(n)  t } \, 
\mb P^N_{A_N} \Big[ H \le N^2  \,,\, D_0 > D_{\ell_N}  
\,, \, A_N(N^2) \in \mf G_N(n,\ell_N) \,,\,
\tau_{n-1} \ge N^2 \Big] \;+\; o_N(1)\;.
\end{equation*}

On the set $\{H \le N^2 \,,\, D_0 > D_{\ell_N} \,,\, \tau_{n-1} \le
N^2 \}$ two particles which were at distance at least $\ell_N$ met in
a time interval of length bounded by $N^2$. Indeed, the time
$\tau_{n-1}$ may correspond to the coalescence of two particles on the
set $B_N$ or one particle in the set $B_N$ and one in the set $\{x_N,
x_N\pm e_j\}$. In both cases, these particles were initially at
distance at least $\ell_N$ from each other. The time $\tau_{n-1}$ may
also correspond to the coalescence of the particles initially at $x_N$,
$x_N\pm e_j$. In this case, at time $H\le N^2\wedge D_0$ these particles
were at distance $\ell_N$.

By Lemma \ref{lem8} with $n=2$, the probability that two particles
which are at distance $\ell_N$ meet before time $N^2$ vanish as
$N\to\infty$. We may therefore remove from the previous probability
the event $\{\tau_{n-1} \ge N^2\}$. We may also remove, as explained
above in the proof, the events $\{H \le N^2\}$ and $\{A_N(N^2) \in \mf
G_N(n,\ell_N)\}$, so that
\begin{equation*}
\mb P^N_{A_N} [ \tau_{n-1} \ge t \theta_N] \;=\;
e^{- \bs \lambda(n)  t } \;
\mb P^N_{A_N} \big[ D_0 > D_{\ell_N} \big] \;+\; o_N(1)\;. 
\end{equation*}
As $N\to\infty$, this latter probability converges to the escape
probability, denoted by $v_d$, which proves the lemma.
\end{proof}


The next result follows from the previous lemma and from the uniform
integrability provided by Lemma \ref{lem11}.

\begin{corollary}
\label{lem18}
Assume that $d\ge 3$ and $n\ge 2$.  Fix a sequence of sets $A_N \in
\ms E_{N}^{n}$ such that $A_N=\{x_N,x_N\pm e_j\}\cup B_N$, where
$B_N\cup\{x_N\}$ belongs to $\mf G_N(n-1,\ell_N)$. Then,
\begin{equation*}
\lim_{N\to\infty} \frac 1{\theta_{N}}\,
\mb E^N_{A_N} [ \tau_{n-1}] \;=\; \frac{v_d}{\bs\lambda(n)}  \;\cdot
\end{equation*}
\end{corollary}

By \eqref{thetan}, the previous limit can be written as
\begin{equation}
\label{feq1}
\lim_{N\to\infty} \frac 1{N^d}\,
\mb E^N_{A_N} [ \tau_{n-1}] \;=\; \frac{1}{2\, \bs\lambda(n)}  \;\cdot
\end{equation}
We turn to the $2$-dimensional case. 

\begin{lemma}
\label{lem13}
Assume that $d=2$ and $n\ge 2$.  Fix a sequence of sets $A_N \in \ms
E_{N}^{n}$ such that $A_N=\{z_N,z_N \pm e_j\}\cup B_N$, where
$B_N\cup\{z_N\}$ belongs to $\mf G_N(n-1,\ell_N)$. Then,
\begin{equation*}
\lim_{N\to\infty} \frac 1{N^2}\,
\mb E^N_{A_N} [ \tau_{n-1}] \;=\; \frac{1}{2\,\bs\lambda(n)}  \;\cdot
\end{equation*}
\end{lemma}

\begin{proof}
Fix a sequence of sets $A_N$ satisfying the hypotheses of the
lemma. Enumerate the points of $A_N=\{x_1, \dots, x_n\}$ in such a
way that $x_1=z_N$, $x_2=z_N\pm e_j$. Denote by $x_i(t)$ the position
at time $t$ of the random walks initially at $x_i$.

Let $(\ell_N :N\ge 1)$, $(m_N : N\ge 1)$ be the sequences $\ell_N =
N/(\log N)^4$, $m_N = N/\log N$. Notice that both sequences fulfill
the conditions above \eqref{13}.  Let $T_{1,2}$ be the first time the
difference $x_1(t) - x_2(t)$ reaches the distance $\ell_N$, $T_{1,2} =
\inf\{t>0 : \Vert x_1(t) - x_2(t)\Vert \ge \ell_N\}$, and denote by
$T_{i}$, $1\le i\le n$, the first time the particle $x_i$ reaches a
distance $m_N$ from its original position: $T_{i} = \inf\{t>0 :
\Vert x_i(t) - x_i(0)\Vert \ge m_N\}$. The proof of the lemma relies
on the estimates \eqref{15}, \eqref{16} and \eqref{17}.

Since the difference $x_1(t) - x_2(t)$ evolves as a random walk
speeded-up by $2$,
\begin{equation*}
\mb E^N_{A_N} \big[ T_{1,2} \big] \;=\; E^N_{e_1} \big[ \bar{D}_{\ell_N}
\big] \;, 
\end{equation*}
where $\bar{D}_{\ell_N}$ is the first time the particle reaches a
distance $\ell_N$ from the origin, and $P^N_{e_1}$ represents the
distribution of a symmetric, nearest-neighbor random walk speeded-up
by $2$, starting from $e_1$. Denote by $B(x,r)$ the ball centered at $x$
of radius $r$. By equation (6.5) in \cite{jlt1} and a simple estimate
of the capacity between $0$ and $B(0,\ell_N)^c$,
$E^N_{e_1}[\bar{D}_{\ell_N}] \le C_0 \ell^2_N$ for some constant $C_0$
independent of $N$. Hence,
\begin{equation}
\label{15}
\lim_{N\to\infty} \frac 1{N^2}\, \mb E^N_{A_N} \big[ T_{1,2} \big]
\;=\;0\;. 
\end{equation}

For every $1\le i\le n$, and every sequence $(S_N)_{N\ge 1}$ of
non-negative numbers, 
\begin{equation*}
\mb P^N_{A_N} [T_i \le S_{N} ] \;=\; 
\widehat{P}^N_{0} [\bar{D}_{m_N} \le S_{N} ]
\;=\;  \widehat{P}^N_{0} 
\big[ \sup_{t\le S_{N}} \Vert x(t) \Vert \ge m_N \big ]\;,
\end{equation*}
where $\widehat{P}^N_{0}$ stands for the distribution of a
nearest-neighbor, symmetric, random walk starting from the origin. The
difference with respect to $P^N_{0}$ is that the random walk is
not speeded-up by $2$ under $\widehat{P}^N_{0}$. An elementary
random walk estimation yields that the right hand side multiplied by
$\log N$ vanishes as $N\to\infty$ if we choose $S_N = N^2/(\log
N)^4$. Hence, wit this definition for $S_N$, for all $1\le i\le n$, 
\begin{equation*}
\lim_{N\to\infty}  (\log N) \, \mb P^N_{A_N} [T_i \le S_N ]
\;=\;0\;.
\end{equation*}
In contrast,
\begin{equation*}
\mb P^N_{A_N} [T_{1,2} \ge S_{N} ] \;=\; 
P^N_{e_1} [\bar{D}_{\ell_N} \ge S_{N} ]\;=\;
P^N_{e_1} \big [\sup_{t\le S_{N}} \Vert x(t) \Vert \le \ell_N \big]\;.
\end{equation*}
Another elementary random walk estimation yields that the right hand
side multiplied by $\log N$ vanishes for the same choice of the
sequence $S_N$. Hence,
\begin{equation*}
\lim_{N\to\infty}  (\log N) \, \mb P^N_{A_N} [T_{1,2} \ge S_N ] 
\;=\;0\;.
\end{equation*}
It follows from the last two estimates that
\begin{equation}
\label{16}
\lim_{N\to\infty}  (\log N) \, \mb P^N_{A_N} \big [T_{1,2} \ge \min_i T_i \big]
\;=\;0\;.
\end{equation}

Denote by $\tau_{i,j}$, $1\le i \not = j\le n$, the first time the
particles $x_i$, $x_j$ meet, $\tau_{i,j} = \inf \{ t>0 : x_i(t) =
x_j(t)\}$.  The arguments used to derive \eqref{16} show that for all
pairs $\{i,j\} \not = \{1,2\}$,
\begin{equation}
\label{17}
\lim_{N\to\infty}  (\log N) \, \mb P^N_{A_N} \big [T_{1,2} \ge \tau_{i,j} \big]
\;=\;0\;.
\end{equation}

We are now in a position to prove the lemma.  By the strong Markov
property,
\begin{align*}
\mb E^N_{A_N} [ \tau_{n-1} ] \; & =\;
\mb E^N_{A_N} \Big[ \, \big[ T_{1,2} + \tau_{n-1} \circ \vartheta_{T_{1,2}} \big] \,
\mf 1\{T_{1,2} < \tau_{n-1} \} \, \Big] 
\; + \;
\mb E^N_{A_N} \Big[ \tau_{n-1} \, \mf 1\{ \tau_{n-1} < T_{1,2} \} \, \Big] \\
& = \;
\mb E^N_{A_N} \Big[ \mb E^N_{A_N(T_{1,2})} \big[ \tau_{n-1} \big]\, 
\mf 1\{T_{1,2} < \tau_{n-1}\} \, \Big] 
\; + \; \mb E^N_{A_N} \Big[ \tau_{n-1} \, \mf 1\{ \tau_{n-1} < T_{1,2} \} \, \Big] \;.
\end{align*}
The second term is bounded by $\mb E^N_{A_N} [ T_{1,2} ]$.  By
\eqref{15}, this expectation divided by $N^2$ vanishes as
$N\to\infty$. On the other hand,
\begin{align*}
& \frac 1{N^2}\, \mb E^N_{A_N} \Big[ \mb E^N_{A_N(T_{1,2})} \big[ \tau_{n-1} \big]\, 
\mf 1\{T_{1,2} \ge \min_i T_i \,,\, T_{1,2} < \tau_{n-1} \} \, \Big]  \\
&\quad \;\le\;
\sup_{A\in \ms E_{N}^{n}} \frac 1{(\log N)\, N^2}\, \mb E^N_{A} \big[ \tau_{n-1} \big]\, 
(\log N) \mb P^N_{A_N} \big[ T_{1,2} \ge \min_i T_i  \big]\;.
\end{align*}
This expression vanishes as $N\to\infty$ because, by Lemma \ref{lem11},
the first term is uniformly bounded and, by \eqref{16}, the second
term tends to $0$.

Up to this point, we proved that
\begin{equation*}
\lim_{N\to\infty} \frac 1{N^2} \mb E^N_{A_N} [ \tau_{n-1} ] \; =\;
\lim_{N\to\infty} \frac 1{N^2}\, \mb E^N_{A_N} \Big[ 
\mb E^N_{A_N(T_{1,2})} \big[ \tau_{n-1} \big]\, 
\mf 1 \big\{T_{1,2} <  \min_i \{\tau_{n-1},T_i\} \big\} \, \Big]\;.
\end{equation*}

On the set $\{T_{1,2} < \min_i T_i \}$, $A_N(T_{1,2})$ belongs to $\mf
G_N(n, \ell_N)$. Hence, by Corollary \ref{lem10} and by
\eqref{thetan},
\begin{equation*}
\frac 1{N^2}\, \mb E^N_{A_N(T_{1,2})} \big[
\tau_{n-1} \big] \;=\; (\log N)\, \frac{\pi^{-1}}{\bs \lambda(n)} \, 
\big[ 1+ o_N(1)\big]\;, 
\end{equation*}
so that
\begin{equation*}
\lim_{N\to\infty} \frac 1{N^2} \mb E^N_{A_N} [ \tau_{n-1} ] \; =\;
\frac{\pi^{-1}}{\bs \lambda(n)} \lim_{N\to\infty} (\log N)
\, \mb P^N_{A_N} \big[ T_{1,2} <  \min_i \{\tau_{n-1},T_i\}  \, \big]\;.
\end{equation*}
By \eqref{16}, in the previous expression we may remove the indicator
of the set $\{T_{1,2} < \min_i T_i\}$. By \eqref{17}, we  may also
exclude the sets $\{\tau _{i,j} \le T_{1,2} \}$ for $\{i,j\} \not =
\{1,2\}$. Hence, the previous expression is equal to
\begin{equation*}
\frac{\pi^{-1}}{\bs \lambda(n)} \lim_{N\to\infty} (\log N)
\, \mb P^N_{A_N} \big[ T_{1,2} <  \tau_{1,2}  \, \big] \;=\; 
\frac{\pi^{-1}}{\bs \lambda(n)} \lim_{N\to\infty} (\log N)
\, P^N_{e_1} \big[ \bar{D}_{\ell_N} < H_0  \, \big]\;,
\end{equation*}
where $H_0$ represents the hitting time of the origin.
By \cite[Lemma 6.10]{jlt1}, the previous expression is equal to $1/[2
\bs \lambda(n)]$, which completes the proof of the lemma.
\end{proof}

Recall the definition of the jump rate $R$ introduced in \eqref{27}.

\begin{lemma}
\label{lem9}
For every $n\ge 2$, 
\begin{equation*}
\lim_{N\to\infty} 
\sum_{A\in \ms E_{N}^{n}} {\pi}_{N}^{n}(A) \, 
\mb E^N_A [ \tau_{n-1}] \,  R(A) \;=\; 1\;.
\end{equation*}
\end{lemma}

\begin{proof}
Since $R(A)=0$ unless $A$ contains two nearest-neighbor points, for
all sets $A$ such that $R(A)>0$, $\mb E^N_A
[ \tau_{n-1}] \le E_{e_1}[H_0]$, where $H_0$ represents the
hitting time of the origin. By \cite[Proposition 10.13]{lpw09}, this
latter expectation is bounded by $C_0 N^d$.

Since $R(A)$ is uniformly bounded, $\mb E^N_A [ \tau_{n-1}] \le
C_0N^d$, and ${\pi}_{N}^{n}(A) = 1/\binom{N^d}{n}$, we may restrict
the sum appearing in the statement of the lemma to sets $A=\{x,x\pm
e_j\}\cup B$, where $B\cup\{x\}\in\mf G_N(n-1,\ell_N)$. The number of
such sets $A$ is $(n-1) d \binom{N^d}{n-1} [1 +o_{N}(1)]$. For them
$R(A) = 1/d$, and, by \eqref{feq1} and Lemma \ref{lem13},
\begin{equation*}
\mb E_{A}^{N}[\tau_{n-1}]\;=\;\frac{N^{d}}{2\bs \lambda(n)}[1 + o_{N}(1)]\;.
\end{equation*}
Hence, the sum alluded to above is equal to
\begin{equation*}
[1 + o_{N}(1)] \, (n-1) \, d \, \frac{\binom{N^d}{n-1}}{\binom{N^d}{n}}
\frac{N^{d}}{2\, \bs \lambda(n)} \, \frac{1}{d}
\;=\;[1 + o_{N}(1)] \, \frac{n(n-1)}{2\, \bs \lambda(n)} \;\cdot
\end{equation*}
The result follows from the definition of $\bs \lambda(n)$ given in
\eqref{14}. 
\end{proof}

\section{Local ergodicity}
\label{sec03}

We prove in this section Proposition \ref{p01}.  It states that we may
replace the time integral of a function $f(A_N(s))$ by the time
integral of a function $F(|A_N(s)|)$. The proof is divided in a
sequence of lemmata.

\begin{lemma}
\label{lem7}
For every $n\ge 2$, there exists a finite constant $C(n)$ such
that
\begin{equation*}
\max_{A\in \ms E_{N}^{n}}  \mb E^N_A \Big[ \int_0^{\infty} R (A_N(s))
\, ds \Big] \;\le\; C(n) \;.
\end{equation*}
\end{lemma}

\begin{proof}
Since $R(B)=0$ if $|B|=1$,
\begin{equation*}
\int_0^{\infty} R (A_N(s)) \, ds \;=\;
\int_0^{\tau_{1}} R (A_N(s)) \, ds \;.
\end{equation*}
It is therefore enough to prove that for each $n\ge 2$, there exists a
finite constant $C(n)$ such that
\begin{equation*}
\max_{A\in \ms E_{N}^{n}}  \mb E^N_A \Big[ \int_0^{\tau_{n-1}} R (A_N(s))
\, ds \Big] \;\le\; C(n) \;.
\end{equation*}

Fix $n\ge 2$ and a set $A=\{x_1, \dots, x_n\}$ in $\ms E_{N}^{n}$. Denote by
$x_i(s)$ the position at time $s$ of the particle $x_i$ and by
$\tau_{i,j}$ the collision time of particles $i$ and $j$: $\tau_{i,j} =
\inf\{t>0 : x_i(t) = x_j(t)\}$. As
\begin{equation*}
\int_0^{\tau_{n-1}} R (A_N(s)) \, ds \;\le\;
\sum_{i\not = j} \int_0^{\tau_{i,j}} 
\mf 1 \big\{|x_i(s) - x_j(s)|=1 \big\} \, ds\;,
\end{equation*}
it is enough to estimate
\begin{equation*}
\mb E^N_{\{x_i,x_j\}}  \Big[ \int_0^{\tau_{i,j}} 
\mf 1 \big\{|x_i(s) - x_j(s)|=1 \big\}
\, ds \Big] \;.
\end{equation*}
As the difference evolves as a random walk speeded-up by $2$, it is
enough to bound, for $x\in \bb T^d_N$,
\begin{equation*}
E^N_{x}  \Big[ \int_0^{H_0} \mf 1 \big\{x(s) = e_1 \big\}
\, ds \Big] \;,
\end{equation*}
where $H_0$ stands for the hitting time of the origin. This integral
represents the time spent at $e_1$ before hitting the origin. In
particular, it is bounded by a geometric sum of independent
exponential random variables, which completes the proof of the lemma. 
\end{proof}

\begin{remark}
\label{obsL1}
It follows from last lemma and the strong Markov property at time
$\tau_{n}$ that there exists a finite constant $C(n)$ such that
\begin{equation*}
\max_{A\in E_{N}}\mb E_{A}^{N}\Big[\int_{0}^{\infty}
R\big(A_{N}(s)\big)\, \mf 1\{|A_{N}(s)|\leq n\}\Big]
\;\leq\; C(n)\;.
\end{equation*}
\end{remark}

Recall the definition of the sequence $a_N$ introduced in \eqref{02},
and that $\pi^n_N$ represents the uniform measure in $\ms E^n_N$. 

\begin{lemma}
\label{a05}
For every $n\ge 2$,
\begin{equation*}
\lim_{N\to\infty} \max_{A\in \mf G_N(n,a_N)}  \Big|\, 
\mb E^N_A \Big[ \int_0^{\tau_{n-1}} R(A_N(s))
\, ds \Big] \, -\,
\sum_{B\in \ms E_{N}^{n}} {\pi}_{N}^{n}(B) \, 
\mb E^N_B [ \tau_{n-1}] \,  R(B) \Big| \;=\; 0 \;.
\end{equation*}
\end{lemma}

\begin{proof}
The goal is to replace the initial condition $A$ by the pseudo-invariant
measure ${\pi}_{N}^{n}$ and then to apply Lemma \ref{lem5}. To carry
out this strategy, we remove from the time integral an interval large
enough for the process to relax and small enough not to interfere with
the overall value of the time integral.

Fix a set $A$ in $\mf G_N(n,a_N)$, enumerate its elements, $A = \{x_1,
\dots, x_n\}$, and denote by $x_i(t)$ the position at time $t$ of the
particle initially at $x_i$. Let $D_1$ be the first time two particles
are at distance $1$ from each other: $D_1 = \inf \{t\ge 0 : \Vert
x_i(t) - x_j(t)\Vert =1 \text{ for some } i\not = j\}$. Note that
$R(A_N(s)) =0$ for $s\le D_1$ and that $D_1 \le \tau_{n-1}$.

Let $\gamma_N$ be the sequence introduced in \eqref{09}. We claim that
\begin{equation*}
\lim_{N\to\infty} \max_{A \in \mf G_N(n,a_N)} 
\mb E^N_A \Big[ \mf 1\{ \tau_{n-1} \le \gamma_N\}
\int_0^{\tau_{n-1}} R(A_N(s)) \, ds \Big] \;=\;0 \;.
\end{equation*}
Indeed, as $R(A_N(s)) =0$ for $s < D_1$ and $D_1 \le \tau_{n-1}$, we
may replace the lower limit in the integral by $D_1$ and include in
the indicator the condition $D_1 \le \gamma_N$ to bound the previous
expectation by
\begin{equation}
\label{21}
\mb E^N_A \Big[ \mf 1\{ D_1 \le \gamma_N\}
\int_{D_1}^{\tau_{n-1}} R(A_N(s)) \, ds \Big]\;.
\end{equation}
By the strong Markov property, this expression is bounded by
\begin{equation*}
\mb P^N_A \big[ D_1 \le \gamma_N \big]\,
\max_{B\in \ms E_{N}^{n}}
\mb E^N_B \Big[ \int_{0}^{\tau_{n-1}} R(A_N(s)) \, ds \Big]\;.
\end{equation*}
By Lemma \ref{lem7} the above expectation is bounded, and by equation
(6.18) in \cite{jlt1} the probability vanishes as $N\to\infty$
uniformly in $A \in \mf G_N(n,a_N)$. Note that in dimension $d\ge 3$,
by equation (6.6) in \cite{jlt1}, the result (6.18) holds for any
sequence $l_N$ such that $1\ll l_N \ll N$. This proves the claim.

Denote by $\vartheta_s: D(\bb R_+, E_N) \to D(\bb R_+, E_N)$,
$s\ge 0$, the time translation operators such that
$(\vartheta_s \omega)(t) = \omega(t+s)$ for all $t\ge 0$.  It follows
from the previous assertion that we may introduce the indicator of the
set $\{\gamma_N < \tau_{n-1}\}$ in the expectation appearing in the
statement of the lemma. After the inclusion in the expectation of the
indicator of the set $\{\gamma_N < \tau_{n-1}\}$, in the upper limit
of the integral rewrite $\tau_{n-1}$ as
$\gamma_N + \tau_{n-1} \circ \vartheta_{\gamma_N}$ and apply the
Markov property to get that the expectation is equal to
\begin{equation}
\label{10}
\begin{aligned}
& \mb E^N_A \Big[ \mf 1\{ \gamma_N < \tau_{n-1} \} 
\int_0^{\gamma_N} R (A_N(s)) \, ds \Big] \\
&\quad \;+\; \mb E^N_A \Big[ \mf 1\{ \gamma_N < \tau_{n-1} \} \,
\mb E^N_{A_N(\gamma_N)} \Big[ \int_0^{\tau_{n-1}} R(A_N(s))
\, ds \Big]\, \Big] \;.
\end{aligned}
\end{equation}

We claim that the first term vanishes as $N\to\infty$, uniformly in
$A\in \mf G_N(n,a_N)$. Recall the definition of the hitting time
$D_1$. If $\gamma_N \le D_1$, the expression inside the expectation
vanishes because $R(A_N(s))=0$ for $s\le D_1$. We may therefore assume
that $D_1\le \gamma_N$. We may also replace the lower limit of the
integral by $D_1$ and the upper limit by $\tau_{n-1}$ to find out that
the first term in \eqref{10} is bounded by \eqref{21}. Since the
expectation in \eqref{21} vanishes as $N\to\infty$, uniformly in $A\in
\mf G_N(n,a_N)$, the claim is proved.

It remains to examine the second expectation in \eqref{10}. To apply
Lemma \ref{lf1}, let $F: \ms E^{\le n}_N \to \bb R$ be the function
given by
\begin{equation}
\label{feq2}
F(B) \;=\; \mb E^N_{B} \Big[ \int_0^{\tau_{n-1}} R(A_N(s))
\, ds \Big] \;,\quad B\in\ms E^n_N \;,
\end{equation}
$F(B) =0$ for $B\not\in \ms E^n_N$.  By Lemma \ref{lem7}, $F$ is
uniformly bounded, $\Vert F \Vert \le C(n)$, and therefore fulfills
the condition of Lemma \ref{lf1}. Hence, by this result, the second
term in \eqref{10} can be written as
\begin{equation*}
\mb E^N_{{\pi}_{N}^{n}} \Big[ \int_0^{\tau_{n-1}} R(A_N(s))
\, ds \Big] \;+\; R_{N}\;,
\end{equation*}
where the absolute value of the remainder $R_{N}$ is bounded by
\begin{equation*}
C(n) \, \Big\{ \, \mb P_{A}^{N} [\tau_{n-1} \le \gamma_{N}]  
\;+\; 2\, \|  p^{(n)}_{\gamma_N}({\bf a}, \cdot) - \pi^{\otimes
  n}_N(\cdot) \|_{\rm TV} \;+\; c_N\Big\}\;.
\end{equation*} 
In this formula, ${\bf a} = (a_1, \dots, a_n)$, $a_j$ are the elements
of $A$ and $c_N$ a constant which vanishes as $N\to\infty$. Since
$\gamma_N \gg t^{N}_{\rm mix}$, the second term inside braces vanishes
as $N\to\infty$, uniformly in $A\in \ms E_{N}^{n}$. By Lemma
\ref{lem8}, the first term inside braces vanishes as $N\to\infty$,
uniformly in $A\in \mf G_N(n,a_N)$. To complete the proof of the
lemma, it remains to apply Corollary \ref{lem5b}.
\end{proof}

\begin{lemma}
\label{a02}
Let $F: \bb N \to \bb R$ be a function which eventually vanishes:
there exists $k_0\ge 0$ such that $F(k)=0$ for all $k> k_0$.
For all $t>0$, $n>1$, 
\begin{equation*}
\lim_{N\to\infty} \max_{A\in \mf G_N(n, a_N)}  \Big|\, \mb E^N_{A}
\Big[ \int_{0}^{t \theta_N}  \big\{ R(A_N(s))
- \theta^{-1}_N \, \mf n_s \big\} \,  F(|A_N(s)|) \, ds  \, \Big]\,
\Big| \;=\; 0 \;.
\end{equation*}
\end{lemma}

\begin{proof}
Fix $n\ge 2$ and $A$ in $\mf G_N(n, a_N)$.  Since $R(A')$,
$\bs \lambda(|A'|)$ vanish for $|A'|=1$, if $k_0\le 1$ there is
nothing to prove. Assume, therefore, that $k_0\ge 2$.  Since $F(k)=0$
for $k> k_0$, we may start the integral from $\tau_{n_0}$, where
$n_0 = n \wedge k_0$.  If $t \theta_N \le \tau_{n_0}$, the integral
vanishes. We may therefore insert inside the expectation the indicator
function of the set $\{t \theta_N > \tau_{n_0}\}$, which can be written
as the disjoint union of the sets $\{\tau_j < t \theta_N \wedge \tau_1
\le \tau_{j-1} \}$, $2\le j\le n_0$. Hence, the time-integral
appearing in the statement of the lemma can be written as
\begin{equation}
\label{08}
\begin{aligned}
& \sum_{j=2}^{n_0} \mf 1\{\tau_j < t \theta_N \wedge \tau_1 \le \tau_{j-1} \}
\int_{\tau_{n_0}}^{\tau_{j-1}}  \widehat{R}(A_N(s))
\,  F(|A_N(s)|) \, ds \\
&\quad - \; \sum_{j=2}^{n_0} \mf 1\{\tau_j < t \theta_N \wedge
\tau_1 \le \tau_{j-1} \}
\int_{t \theta_N}^{\tau_{j-1}}  \widehat{R}(A_N(s))
\,  F(|A_N(s)|) \, ds \; , 
\end{aligned}
\end{equation}
where $\widehat{R}(A) = R(A) - \theta^{-1}_N \, \bs \lambda(|A|)$. 

We consider each term separately. Write the integral appearing in the
first line as a sum of integrals on the intervals $[\tau_i,
\tau_{i-1})$ and sum by parts to obtain that the first expression is
equal to
\begin{equation*}
\sum_{i=2}^{n_0} \mf 1\{\tau_i < t \theta_N \wedge
\tau_1 \le \tau_{1} \} \,  F(i) \,
\int_{\tau_i}^{\tau_{i-1}}  \widehat{R}(A_N(s))\, ds  \;,
\end{equation*}
where we used the fact that $F$ is constant in the time interval
$[\tau_i, \tau_{i-1})$.  Remove from the indicator the condition $\{t
\theta_N \wedge \tau_1 \le \tau_{1}\}$, which is always satisfied, and
replace $\{\tau_i < t \theta_N \wedge \tau_1\}$ by $\{\tau_i < t
\theta_N\}$. Fix $2\le i\le n$, disregard the constant $F(i)$, and
consider the expectation with respect to $\mb P^N_A$:
\begin{equation}
\label{19}
\mb E^N_A \Big[ \mf 1\{\tau_i < t \theta_N \} \,
\int_{\tau_i}^{\tau_{i-1}}  \widehat{R}(A_N(s))\, ds  \Big] \;.
\end{equation}

We claim that
\begin{equation*}
\lim_{N\to\infty} \mb E^N_A \Big[ \mf 1\{ A_N(\tau_i) \not \in \mf G_N (i,a_N)\} \,
\int_{\tau_i}^{\tau_{i-1}}  \big\{ R(A_N(s)) - 
\theta^{-1}_N \mf n_s\big\}\, ds  \Big] \;=\; 0\;.
\end{equation*}
Indeed, by the strong Markov property, the absolute value of the previous
expectation is less than or equal to
\begin{equation*}
\mb P^N_A \big[ A_N(\tau_i) \not \in \mf G_N (i,a_N)\, \big] \,
\max_{B\in \ms E_{N}^{i}} \Big\{
\mb E^N_B \Big[ \int_{0}^{\tau_{i-1}}  R(A_N(s))\, ds  \Big] 
\,+\, \bs \lambda(i)\,\theta^{-1}_N  \mb E^N_B \big[ \tau_{i-1} \big]  \Big\}\;.
\end{equation*}
By Lemmata \ref{lem11} and \ref{lem7}, the maximum is bounded.  On the
other hand, since $A$ belongs to $\mf G_N(n,a_N)$, by Proposition
\ref{lem3}, the probability vanishes as $N\to\infty$, which proves the
claim.

We may therefore insert in \eqref{19} the indicator of the set $\{
A_N(\tau_i) \in \mf G_N (i,a_N)\}$.  By the strong Markov property,
this expectation is equal to
\begin{equation*}
\mb E^N_A \Big[ \mf 1\{\tau_i < t \theta_N \,,\,
A_N(\tau_i) \in \mf G_N (i,a_N) \} \, \mb E^N_{A_N(\tau_i)} \Big[
\int_{0}^{\tau_{i-1}}  \widehat{R}(A_N(s))\, ds \Big]\,  \Big] \;.
\end{equation*}
By Lemmata \ref{a05} and \ref{lem9}, 
\begin{equation*}
\lim_{N\to\infty} \mb E^N_{B} \Big[ \int_{0}^{\tau_{i-1}}  R(A_N(s))\,
ds \Big] \;=\; 1
\end{equation*}
uniformly for $B\in \mf G_N (i,a_N)$. By Corollary \ref{lem10}, as
$N\to\infty$, $\bs \lambda(i)\, \mb E^N_{B} [ \tau_{i-1}/\theta_N]$
converges to $1$ uniformly for $B\in \mf G_N (i,a_N)$.

It remains to examine the second expression in \eqref{08}. The
argument is similar to the one presented above.  Fix $2\le j\le n_0$
and take the expectation with respect to $\mb P^N_A$ for
$A\in \mf G_N (n,a_N)$. Since $\tau_1 \ge \tau_j$, we may remove
$\tau_1$ from the indicator. For $j=2$ the set becomes
$\{\tau_2 < t \theta_N\}$, while for $2<j \le n_0$ it is given by
$\{ \tau_j < t \theta_N \le \tau_{j-1}\}$. In the first case, to
uniform the notation, we insert the condition
$t \theta_N \le \tau_{1}$. This is possible because the integral
vanishes if this bound is not fulfilled.

We claim that
\begin{equation*}
\lim_{N\to\infty} \mb E^N_A \Big[ \mf 1\{ \ms G_N \} \,
\int_{t\theta_N}^{\tau_{j-1}}  \big\{ R(A_N(s)) 
- \theta_N^{-1}\mf n_s\big\}\, ds  \Big] \;=\; 0\;,
\end{equation*}
where $\ms G_N$ is the set $\{ \tau_j < t \theta_N \le \tau_{j-1}
\,,\, A_N(t\theta_N) \not \in \mf G_N (j,a_N)\}$. The proof of this
claim is identical to the one produced below \eqref{19}. Observe that
on the set $\{ \tau_{j-1} \ge t \theta_N\}$ we may write $\tau_{j-1}$
as $t\theta_N + \tau_{j-1} \circ \vartheta_{t\theta_N}$. Apply the
Markov property at time $t\theta_N$, estimate the conditional
expectation by the supremum over all sets in $\ms E_{N}^{j}$, and apply
Lemmata \ref{lem11} and \ref{lem7}, and Lemma \ref{lem16} (instead
of Proposition \ref{lem3}).

After inserting in the expectation the indicator of the set
$\{A_N(t\theta_N) \in \mf G_N (j,a_N)\}$, applying the
Markov property at time $t\theta_N$, the expectation becomes
\begin{equation*}
\mb E^N_A \Big[ \mf 1\{\ms M_N \} \, \mb E^N_{A_N(t\theta_N)} \Big[
\int_{0}^{\tau_{j-1}}  \widehat{R}(A_N(s))\, ds \Big]\,  \Big] \;,
\end{equation*}
where $\ms M_N = \{ \tau_j < t \theta_N \le \tau_{j-1} \,,\,
A_N(t\theta_N) \in \mf G_N (j,a_N)\}$. By the first part of the proof,
this expression vanishes as $N\to\infty$.
\end{proof}

\begin{proof}[Proof of Proposition \ref{p01}]
Fix $\varepsilon>0$. In view of Proposition \ref{p02}, choose
$M\in\mathbb{N}$ such that
$\mb P^{N}[\, |A_{N}(t_{0}\theta_{N})|> M]\leq\varepsilon$.  Let
$W(A) =\{ \theta_{N}\,R(A) - \bs \lambda(|A|) \} \, F(|A|)$. There
exists a finite constant $C(F,B,t)$ such that
\begin{equation}
\label{lf4}
\Big|\,\mb E^{N}\Big[B^{N}\mf 1\{|A_{N}(t_{0}\theta_{N})|> M\}
\int_{t_{0}}^{t}W\big(A_{N}(s\theta_{N})\big)\, ds\Big]\,\Big|
\,\leq\,C(F,B,t)\,\varepsilon\;.
\end{equation}
To prove this assertion, apply the Markov property to write the
expectation appearing in the left-hand side as
\begin{equation*}
\mb E^{N}\Big[B^{N}\mf 1\{|A_{N}(t_{0}\theta_{N})|> M\}\,
\mb E_{A(t_{0}\theta_{N})}^{N}\Big[
\int_{0}^{t-t_{0}}W\big(A_{N}(s\theta_{N})\big)\, ds\Big]\, \Big]\;.
\end{equation*}

We claim that the absolute value of the expectation with respect to
$\mb P_{A(t_{0}\theta_{N})}^{N}$ is bounded by a constant depending on
$F$ and $t$. On the one hand, the function
$ \bs \lambda(|A|) \, F(|A|)$ is bounded because $F(k)=0$ for all $k$
large enough. On the other hand, since $F$ vanishes outside a finite
subset of $\bb N$, by Remark \ref{obsL1}, the expectation of the time
integral of
$\theta_{N}\,R(A_{N}(s\theta_{N}))\, F(A_{N}(s\theta_{N}))$ is
bounded. This proves the claim.

It follows from this claim that the absolute value of the expectation
appearing in the last displayed equation is bounded by
\begin{equation*}
C(F,B,t) \, \mb P^{N} \big[\, |A_{N}(t_{0}\theta_{N})|> M\, \big]\;,
\end{equation*}
Assertion \eqref{lf4} follows from the choice of $M$.

A similar argument, using Corollary \ref{obsaN} instead of Proposition
\ref{p02}, proves that for all $N$ sufficiently large
\begin{multline*}
\Big|\,\mb E^{N}\Big[B^{N}\mf 1\Big\{A_{N}(t_{0}\theta_{N})
\not\in\ms E_{N}^{1} \, \cup \,
\bigcup_{k=2}^{N^{d}} \mf G_N(k, a_N)\Big\}\\
\times\,\int_{t_{0}}^{t}W\big(A_{N}(s\theta_{N})\big)\, ds\Big]\,\Big|
\,\leq\,C(F,B,t)\,\varepsilon\;.
\end{multline*}

It follows from the previous two estimates that we may restrict our
attention to the expectation
\begin{equation*}
\mb E^{N}\Big[B^{N}\mf 1\big\{ \ms M_N(M,t_{0}) \big\}
\int_{t_{0}}^{t}W\big(A_{N}(s\theta_{N})\big)\, ds\Big]\,,
\end{equation*}
where
$\ms M_N(M,t_{0})=\big\{|A_{N}(t_{0}\theta_{N})|\leq
M\,,\,A_{N}(t_{0}\theta_{N})\in \ms E_{N}^{1} \, \cup
\,\bigcup_{k=2}^M \mf G_N(k, a_N)\big\}$.
Applying the Markov property at time $t_{0}\theta_{N}$ yields that the
absolute value of the previous expectation is bounded by
\begin{equation*}
C(B)\max_{A\in \ms E_{N}^{1} \, \cup \,\bigcup_{k=2}^M \mf G_N(k, a_N)}\Big|\mb E_{A}^{N}
\Big[\int_{0}^{t-t_{0}}W\big(A_{N}(s\theta_{N})\big)\, ds\Big]\Big|\;,
\end{equation*}
where the constant $C(B)$ is an upper bound for
$(|B^{N}|:N\in\mathbb{N})$. This expression vanishes as $N\to\infty$
by Lemma \ref{a02}, which completes the proof of the proposition.
\end{proof}

\subsection{Equilibrium expectation of hitting times}
We conclude this section with a result on the equilibrium expectation
of hitting times.  Let $X_t$ be a reversible, irreducible,
continuous-time Markov chain on a finite set $E$. Denote by $\pi$ the
unique stationary state and by $H_B$, $B\subset E$, the hitting time
of the set $B$: $H_B = \inf\{t\ge 0 : X_t\in B\}$. Denote by $\bb P_x$
the distribution of the Markov chain $X_t$ starting from
$x$. Expectation with respect to $\bb P_x$ is represented by $\bb
E_x$. As usual, for a probability measure $\mu$ on $E$, $\bb P_\mu =
\sum_{x\in E} \mu(x)\, \bb P_x$.

\begin{lemma}
\label{lem5}
For all subsets $B$ of $E$, and all functions $f:E \to \bb R$, 
\begin{equation*}
\bb E_\pi \Big[ \int_0^{H_B} f(X_s) \, ds \Big] \;=\;
\sum_{x\in E} \pi(x)\, f(x)\, \bb E_x [ H_B ]\;.
\end{equation*}
\end{lemma}

\begin{proof}
Denote by $(Y_k)_{k\ge 0}$ the skeleton of the chain $X_t$. This is
the discrete-time Markov chain which keeps track of the sequence of
elements of $E$ visited by the process. Denote by $\lambda(x)$, $x\in
E$, the holding time at $x$. Representing the process $X_t$ in terms
of the chain $Y_k$ and independent, mean-one, exponential random
variables (cf. Section 6 of \cite{bl2}), the expectation appearing in
the statement of the lemma can be written as
\begin{equation*}
\bb E_\pi \Big[ \sum_{k=0}^{h_B-1}  \frac{f(Y_k)}{\lambda(Y_k)}  \Big] 
\;=\; \sum_{k\ge 0} \sum_{x\not\in B} \sum_{y \not\in B} \pi(x)\, 
\frac{f(y)}{\lambda(y)} \, \bb P_x \big[Y_k=y \,,\, h_B>k \big]\;,
\end{equation*}
where $h_B$ stands for the hitting time of the set $B$ by the Markov
chain $Y_k$: $h_B = \min\{j\ge 0 : Y_j\in B\}$.  By reversibility, the
previous expression is equal to
\begin{equation*}
\sum_{k\ge 0} \sum_{x\not\in B} \sum_{y \not\in B}  \pi(y)\, 
\frac{f(y)}{\lambda(x)} \, \bb P_y \big[Y_k=x \,,\, h_B>k \big]
\;=\; \sum_{y \in E} \pi(y) \, f(y) \,
\bb E_y \Big[ \sum_{k=0}^{h_B-1}  \frac{1}{\lambda(Y_k)}  \Big] \;. 
\end{equation*}
The last expectation is equal to $\bb E_y [ H_B]$, which completes the
proof of the lemma.
\end{proof}

\begin{corollary}
\label{lem5b}
For every $n\ge 2$,
\begin{equation*}
\lim_{N\to\infty} \Big|\, 
\mb E^N_{\pi^n_N} \Big[ \int_0^{\tau_{n-1}} R(A_N(s))
\, ds \Big] \, -\,
\sum_{B\in \ms E_{N}^{n}} {\pi}_{N}^{n}(B) \, 
\mb E^N_B [ \tau_{n-1}] \,  R(B) \, \Big| \;=\; 0 \;.
\end{equation*}
\end{corollary}

\begin{proof}
Let $F: \ms E^{\le n}_N \to \bb R$ be the function given by
\eqref{feq2}, and recall that it is uniformly bounded.  The
expectation appearing in the statement of the lemma is equal to
$E_{\pi^n_N} [F]$. By \eqref{feq3} and since $F$ vanishes on $\ms
E^m_N$, $m<n$, and is uniformly bounded, this expectation is equal to
$E_{\pi^{\otimes n}_N} [F(\{{\bf x} \})] + c_N$, where $\lim_N c_N = 0$.

By definition of $F$,
\begin{equation*}
E_{\pi^{\otimes n}_N} [F(\{{\bf x} \})] \;=\; 
\sum_{{\bf x} \in [\bb T^d_N]^n} \pi^{\otimes n}_N({\bf x}) 
\, \mf 1\{\{ {\bf x} \} \in \ms E^n_N \} \,
\mb E^N_{\{{\bf x} \}} \Big[ \int_0^{\tau_{n-1}} R(A_N(s)) \, ds \Big] \;.
\end{equation*}
Up to time $\tau_{n-1}$ the evolution of $A_N(s)$ corresponds to the
evolution of $n$ independent particles. We may thus replace $A_N(s)$
by $\{ {\bf x}^n_N (s) \}$ inside the expectation, where $\tau_{n-1}$
represents in this context the first time two particles meet. The
previous sum is thus equal to
\begin{equation*}
\sum_{{\bf x} \in [\bb T^d_N]^n} \pi^{\otimes n}_N({\bf x}) 
\, \mf 1\{\{ {\bf x} \} \in \ms E^n_N \} \,
\widetilde{\mb E}^N_{{\bf x}} \Big[ \int_0^{\tau_{n-1}} 
R(\{{\bf x}^n_N (s)\}) \, ds \Big] \;,
\end{equation*}
where $\widetilde{\mb P}^N_{{\bf x}}$ represents the distribution of
${\bf x}^n_N$ starting from ${\bf x}$.

Since $\tau_{n-1}=0$ if the process ${\bf x}^n_N (s)$ starts from a
configuration ${\bf x}$ such that $\{ {\bf x} \} \not \in \ms E^n_N$,
we may remove the indicator in the previous sum. As the process is
reversible and $\pi^{\otimes n}_N$ is its unique stationary state, by
Lemma \ref{lem5}, the sum is equal to
\begin{equation*}
\sum_{{\bf x} \in [\bb T^d_N]^n} \pi^{\otimes n}_N({\bf x})  \, 
\widetilde{\mb E}^N_{{\bf x}} [ \tau_{n-1}] \,  R(\{ {\bf x} \}) \;.
\end{equation*}

As $\tau_{n-1}=0$ if the process ${\bf x}^n_N (s)$ starts from a
configuration ${\bf x}$ such that $\{ {\bf x} \} \not \in \ms E^n_N$,
we may restrict the sum to configurations ${\bf x}$ such that $\{ {\bf
  x} \}\in \ms E^n_N$. For such a configuration, $\widetilde{\mb
  E}^N_{{\bf x}} [ \tau_{n-1}] = {\mb E}^N_{\{{\bf x}\}} [
\tau_{n-1}]$. Hence, the last sum is equal to
\begin{equation*}
\sum_{{\bf x} \in [\bb T^d_N]^n} \pi^{\otimes n}_N({\bf x})  \, 
{\mb E}^N_{\{{\bf x}\}} [\tau_{n-1}] \,  R(\{ {\bf x} \}) 
\;=\; \sum_{A\in \ms E_{N}^{n}} {\mb E}^N_{A} [\tau_{n-1}] \,  R(A) \,
\sum_{\{{\bf x}\} = A} \pi^{\otimes n}_N({\bf x})  \;,
\end{equation*}
where the last sum is performed over all configuration ${\bf x} \in
[\bb T^d_N]^n$ such that $\{{\bf x}\} = A$. Comparing $\sum_{\{{\bf
    x}\} = A} \pi^{\otimes n}_N({\bf x})$ with $\pi^n_N(A)$ yields
that the previous sum is equal to
\begin{equation*}
\big( 1 \,+\, O ( N^{-d})\, \big)
\sum_{A\in \ms E_{N}^{n}} {\mb E}^N_{A} [\tau_{n-1}] \,  R(A) \, \pi^n_N(A)\;,
\end{equation*}
where $O ( N^{-d})$ is a sequence of numbers whose absolute value is
bounded by $C_0 N^{-d}$ for some finite constant $C_0$.  By Lemma
\ref{lem9}, the sum converges to $1$. In particular, the term
$O(N^{-d})$ times the sum is negligible. This completes the proof of
the corollary.
\end{proof}

\section{Proof of Theorem \ref{th1}}
\label{teo12}

The proof of Theorem \ref{th1} is divided in two steps. We show in
Lemma \ref{lem1} that the sequence $(\ms P^N)_{N}$ is tight, and in
Lemma \ref{lem15} that all limit points solve the $\big(C^{1}(S),\ms
L\big)$-martingale problem introduced in Proposition \ref{prop1}.

Denote by $\bb P^N_A$, $A\in E_N$, the probability
measure on $D(\bb R_{+}, E_N)$ induced by the Markov chain $A_N(t)$
\emph{speeded-up by $\theta_N$} starting from $A$. When $A=\bb T^d_N$,
we denote $\bb P^N_A$ simply by $\bb P^N$. Expectation with respect to
$\bb P^N_A$, $\bb P^N$ are represented by $\bb E^{N}_{A}$ and $\bb E^{N}$,
respectively. Note that
\begin{equation}
\label{26}
\ms P^N \;=\; \bb P^N \, \circ\, \widehat{\Psi}_N^{-1} \;,
\end{equation}
where $\widehat{\Psi}_N : D(\bb R_{+}, E_N) \to D(\bb R_{+}, S)$ is given
by $[\widehat{\Psi}_N (\omega)](t) = \Psi_N \big(\omega(t)\big)$.

In the next lemmata, expectation with respect to $\ms P^N,\;\ms P$ are
represented by $E_{\ms P^{N}},\; E_{\ms P}$, respectively.

\begin{lemma}
\label{lem15}
Let $\ms P$ be a limit point of the sequence $(\ms P^N)_{N}$, and let
$f: \bb R \to \bb R$ be a function in $C^1$ which is constant in a
neighborhood of the origin: there exists $\delta>0$ such that
$f(x)=f(0)$ for $x\le \delta$. Then, under $\ms P$, the process
defined by \eqref{24} is a martingale.
\end{lemma}

\begin{proof}
Assume without loss of generality that $(\ms P^N)_{N}$ converges to
$\ms P$.  Let $f: \bb R \to \bb R$ be a function in $C^1$ which is
constant in a neighborhood of the origin. Denote by $M_N(t)$ the $\bb
P^N$-martingale given by
\begin{equation*}
f(\bb X_N(t)) \,-\, f(\bb X_N(0)) \,-\,
\int_0^t \theta_N \, (L_N f)(\Psi_N(A_N(s\theta_N)))\, ds\;,
\end{equation*}
where $\bb X_N(t) = \Psi_N(A_N(t\theta_N))$. Since
\begin{equation*}
(L_N f)(\Psi_N(A)) \;=\; R(A) \, \Big\{ f \Big(\frac x{1-x} \Big) - f(x)
\Big\}\;, 
\end{equation*}
where $x=\Psi_N(A)$, and $R(A)$ is the jump rate introduced in
\eqref{27}, the martingale $M_N(t)$ can be written as
\begin{equation*}
f(\bb X_N(t)) \,-\, f(\bb X_N(0)) \,-\, \theta_N\, 
\int_0^{t} \, R(A_N(s\theta_N))\,
\Big\{ f \Big(\frac {\bb X_N(s)}{1- \bb X_N(s)} \Big) 
- f(\bb X_N(s)) \Big\} \, ds\;.
\end{equation*}

Fix $0\le t_0$, $k\ge 1$, $0\le s_1< \dots < s_k \le t_0$, a
bounded function $G: \bb R^k \to \bb R$, and let $B^{N} = G(\bb X_N(s_1),
\dots, \bb X_N(s_k))$. Since $M_N$ is a martingale, for every $t_0\le t$,
\begin{equation*}
\bb E^N \Big[ B^{N} \, \big\{M_N(t)-M_N(t_0)\big\}\, \Big]\;=\; 0\;.
\end{equation*}
By Proposition \ref{p01}, in the integral part of the martingale we
may replace the rate $\theta_N R(A_N(s\theta_N))$ by $\bs
\lambda(|A_N(s\theta_N)|) = \bs r (\bb X_N(s))$ to obtain that
\begin{equation}
\label{25}
\lim_{N\to\infty} \bb E^N \Big[ B^{N} \, 
\big\{ \widehat{M}_N(t)-\widehat{M}_N(t_0)\big\}\, \Big]\;=\; 0\;,
\end{equation}
where 
\begin{align*}
\widehat{M}_N(t) \;=\;
f(\bb X_N(t)) \,-\, f(\bb X_N(0)) \,-\,
\int_0^{t} \, \bs r(\bb X_N(s))\,
\Big\{ f \Big(\frac {\bb X_N(s)}{1-\bb X_N(s)} \Big) - f(\bb X_N(s)) \Big\} \, ds
\;.
\end{align*}
Notice that the process $\widehat{M}_N(t)$ is expressed as a function
of $\bb X_N$. Therefore, in view of \eqref{26}, we may replace in
\eqref{25} the probability $\bb P^N$ by $\ms P^N$ and write
\begin{equation*}
\lim_{N\to\infty} E_{\ms P^{N}} \Big[ B^{N} \, 
\big\{ \widehat{M}_N(t)-\widehat{M}_N(t_0)\big\}\, \Big]\;=\; 0\;, 
\end{equation*}
Since, by assumption, $(\ms P_N)_{N}$ converges to $\ms P$, 
\begin{equation*}
E_{\ms P} \Big[ B^{N} \, 
\big\{ \widehat{M}_N(t)-\widehat{M}_N(t_0)\big\}\, \Big]\;=\; 0\;.
\end{equation*}
This shows that \eqref{24} is a martingale under $\ms P$ and completes
the proof of the lemma.
\end{proof}

We turn to the tightness of $(\ms P^{N})_{N}$.  Remember that for
$w\in D(\bb R_{+},S)$, the \emph{modified modulus of continuity} is
defined as
\begin{equation*}
\tilde{\omega}(w,t,\delta)\;:=\;\inf_{\Delta}\,\max_{k}\,
\sup_{t_{k}\leq r,s<t_{k+1}}\|w(s)-w(r)\|\;,\quad t>0\;,\quad \delta>0\;,
\end{equation*}
where the infimum extends over all partitions $\Delta=\{0=t_{0}<t_{1}<
\dots<t_{\ell}<t\}$ such that $t_{k+1}-t_{k}\geq \delta$ for
$k=1,\dots,\ell-1$. It is well known (see for instance \cite[Theorem
4.8.1]{k11}) that the tightness follows from
\begin{enumerate}
\item\label{tig1} for any $t\in\bb R_{+}$, the sequence $\big(\bb
  X_{N}(t)\big)_{N}$ is tight in $S$; and

\item\label{tig2} for all $\varepsilon>0$, $t>0$,
\begin{equation}\label{cond2}
\lim_{\delta\to0}\,\sup_{N}\,\ms P^{N}[\tilde{\omega}(\bb X_{N},t,\delta)>\varepsilon]\;=\;0\;.
\end{equation}
\end{enumerate}

Since $\bb X_{N}(t)\in S$ for all $t\in\bb R_{+}$ and $S$ is compact, condition \eqref{tig1}
holds immediately thanks to Prohorov's criterion. Denote by $\sigma_j$, $j\ge 1$,
the hitting time of $1/j$: $\sigma_j = \inf\{t\ge 0: \bb X(t) = 1/j\}$.

\begin{lemma}\label{L4.1}
Condition \eqref{tig2} follows from
\begin{equation}
\label{05}
\lim_{\delta\to 0}
\limsup_{N\to\infty} \ms P^N \big[ \sigma_{j-1}-\sigma_j \le \delta \,
\big] \;=\; 0\;,\quad\forall\,j\geq2\;.
\end{equation}
\end{lemma}

\begin{proof}
Assume that \eqref{05} holds, fix $\varepsilon>0$, $t>0$, $\eta>0$ and
choose $n\in\mathbb{N}$ such that $1/n\leq\varepsilon$. By
Proposition \ref{p02} and by the Markov inequality
\begin{equation*}
\ms P^{N}[\bb X_{N}(t)\leq 1/n]=\;\mb P^{N}[\,|A_{N}(t\theta_{N})|\geq n\,]\;
\leq\;\frac{\mb E^{N}[\, |A_{N}(t\theta_{N})|\, ]}{n}\;
\leq\;\frac{C(t,d)}{n}\;,
\end{equation*}
where $C(t,d)$ is a positive constant depending only on $t$ and
$d$. Then, increasing $n$ if necessary, we can assume that
\begin{equation*}
\ms P^{N}[\sigma_{n}<t]\;>\;1-\eta/3\;.
\end{equation*}

Our assumption implies that there are $\delta_{0}>0$ and
$M\in\mathbb{N}$ such that
\begin{equation*}
\ms P^{N}[\sigma_{j-1}-\sigma_j \geq \delta_{0},\;
\text{for all }j\in\{2,3,\dots, n\}]\;>\;1-\eta/3,
\quad\forall\,N>M.
\end{equation*}
Let $m:=\min\{j\ge 1 :\sigma_{j}<t\}$. On the set
$\{\sigma_{n}<t\}$, define the random partition
$\Delta:=\{0=t_{0}<t_{1}=\sigma_{n}<\dots<t_{\ell}=\sigma_{m}<t\}$. Since
$\bb X_{N} (r)$ is constant in the intervals $[\sigma_{j},
\sigma_{j-1})$, using this partition we deduce that
\begin{equation*}
\tilde{\omega}(\bb X_{N},t,\delta)\leq1/n\leq\varepsilon\;,
\quad\forall\,\delta<\delta_{0}\;, \quad N>M\,,
\end{equation*}
on the event
\begin{equation*}
\{\sigma_{n}<t\}\cap\big\{\sigma_{j-1}-\sigma_j \geq \delta_{0},\;
\text{for all }j\in\{2,3,\dots, n\}\big\}\;,
\end{equation*}
that has probability at least $1-2\eta/3$. Hence
\begin{equation*}
\sup_{N>M}\,\ms P^{N}[\tilde{\omega}(\bb X_{N},t,\delta)>\varepsilon]<2\eta/3\;,\quad
\forall\delta<\delta_{0}\;.
\end{equation*}
On the other hand, it is clear that there is $\delta_{1}>0$ such that
\begin{equation*}
\ms P^{N}[\tilde{\omega}(\bb X_{N},t,\delta)>\varepsilon]
\;<\;\eta/3\;,\quad N\leq M\;,\quad\forall
\delta<\delta_{1}.
\end{equation*}
Therefore
\begin{equation*}
\sup_{N}\,\ms P^{N}[\tilde{\omega}(\bb
X_{N},t,\delta)>\varepsilon]<\eta\;,
\quad\forall\,\delta<
\min\{\delta_{0},\delta_{1}\}\;,
\end{equation*}
which completes the proof, since $\eta>0$ was arbitrary.
\end{proof}

We complete the proof of the tightness in the next lemma.

\begin{lemma}
\label{lem1}
The sequence of measures $(\ms P^N) _{N}$ is tight. 
\end{lemma}

\begin{proof}
By Lemma \ref{L4.1} it is enough to show \eqref{05}. In terms of the
measure $\mb P^N$, the probability appearing in \eqref{05} can be
rewritten as
\begin{equation*}
\mb P^N \big[ \tau_{j-1}-\tau_j \le \delta\, \theta_N \,
\big]\;.
\end{equation*}

Fix $\epsilon >0$ and $M>j$. In view of \eqref{04}, choose $\alpha >0$
small enough for $\mb P^N [ \tau_M \le 3\alpha\theta_N ] \le \epsilon$
for all $N$ sufficiently large. By Proposition \ref{p02}, choose $K\ge
M$ such that $\mb P^N [ \, |A_N(\alpha \theta_N)| \ge K ] \le \epsilon$
for all $N$ sufficiently large. Hence, the probability appearing in
\eqref{05} is less than or equal to
\begin{equation*}
\mb P^N \Big[\, |A_N(\alpha \theta_N)| \le K \,,\,
\tau_M \ge 3\alpha \theta_N \,,\,
\tau_{j-1}-\tau_j \le \delta \theta_N \, \Big]
\;+\; 2\, \epsilon \;.
\end{equation*}
By Lemma \ref{lem16}, this expression is less than or equal to
\begin{equation*}
\mb P^N \Big[\, |A_N(\alpha \theta_N)| \le K \,,\, A_N(2\alpha \theta_N) \in \mf G_N
\,,\, \tau_M \ge 3\alpha \theta_N \,,\,
\tau_{j-1}-\tau_j \le \delta \theta_N \, \Big]
\;+\; 3\, \epsilon \;.
\end{equation*}
By the Markov property, this sum is bounded by
\begin{equation*}
\max_{M \le n \le K} \max_{A\in \mf G_N(n,a_N)} \mb P^N_A \big[\, 
\tau_{j-1}-\tau_j \le \delta \theta_N \, \big]
\;+\; 3\, \epsilon \;.
\end{equation*}
By Propositions \ref{lem3}, \ref{lem4} and the strong Markov property
at time $\tau_j$, the first term of the previous expression vanishes
as $N\uparrow\infty$ and $\delta\to 0$.
\end{proof}

\section{Uniqueness}
\label{sec5}

In order to state the uniqueness result as it has been used in Section
\ref{teo12} we need to introduce the subset $\dom\subseteq C^1(S)$ of
functions $f: S\to \bb R$ which are constant on a neighborhood of
zero: $f\in \dom$ if and only if for some $k(f)\in \bb N$ we have
\begin{equation*}
f(0)=f(1/k) \;, \quad \forall k > k(f)\;.
\end{equation*}
We shall say that a probability measure $\ms P$ on the measurable space $(D(\bb R_+,S),\mc G_{\infty})$ is a solution of the $(\dom,\ms L)$ (resp. $(C^1(S),\ms L)$)-martingale problem  if
\begin{equation}\label{mm1}
M^f_t \;:=\; f(X_t) - \int_0^t (\ms Lf)(X_s)\,ds \;,\quad t\ge 0 
\end{equation}
is a $\ms P$-martingale for every $f\in \dom$ (resp. $f\in C^1(S)$). In addition, we say that $\ms P$ is starting at $x\in S$ whenever $\ms P\{X_0=x\}=1$. 

\subsection{Uniqueness on $S\setminus \{0\}$}
For each $k\in \bb N$, let $\ms P_{1/k}$ be the law on $(D(\bb R_+,S),\mc G_{\infty})$ of a Markov process on $S$ starting at $1/k$ and with transition rates 
$$
q\left(\frac1 n,\frac 1{n-1}\right) \;=\; {n \choose 2}\;, \quad \textrm{for $2\le n\le k$}
$$
and zero elsewhere. By Dinkyn's martingales, the process
\begin{equation}\label{lam1}
f(X_t) - \int_0^t (\ms L^{k}f)(X_s)\,ds \;,\quad t\ge 0 
\end{equation}
is a $\ms P_{1/k}$-martingale, for all $f:S\to \bb R$, where
\begin{equation*}
\ms L^kf(x) \;:=\;
\begin{cases}
{n \choose 2} \left\{ f\left( \frac{1}{n-1}\right) - f\left( \frac{1}{n}\right) \right\}\,, & \textrm{if $ x = \frac{1}{n} \in [\frac{1}{k}, \frac{1}{2}]$}\;, \\
0\,, & \textrm{otherwise} \;.
\end{cases}
\end{equation*}
In particular, $\ms P_{1/k}$ is a solution of the $(\dom,\ms L^k)$-martingale problem. Moreover, uniqueness for this problem can be obtained by standard methods so that
\begin{remark}\label{obs9}
For each $k\in \bb N$, $\ms P_{1/k}$ is the unique solution of the $(\dom,\ms L^k)$-martingale problem starting at $1/k$.
\end{remark}

Since $\ms P_{1/k}\{X_t \ge 1/k\,,\; \forall \, t\ge 0\}=1$ and
\begin{equation}\label{obs2}
\ms L^kf(x) \;=\; \ms Lf(x)\;, \quad \textrm{for all $x\ge 1/k$}
\end{equation}
we may then replace $\ms L^k$ by $\ms L$ in \eqref{lam1}. Therefore,

\begin{remark}\label{obs10}
For each $x\in S\setminus \{0\}$, $\ms P_{x}$ is a solution of the $(C^1(S),\ms L)$, and so, also the $(\dom,\ms L)$-martingale problem.
\end{remark}

We now prove that, for all $x\in S\setminus \{0\}$, $\ms P_x$ is actually the unique solution for both martingale problems when starting at $x$. Of course, it is enough to prove this assertion for $(\dom,\ms L)$. In virtue of Remark \ref{obs9}, it suffices to prove that under any such solution $X_t \ge 1/k$, $\forall \, t\ge 0$ almost surely.

\begin{lemma}\label{unik}
For each $x\in S\setminus \{0\}$, $\ms P_{x}$ is the unique solution of the $(\dom,\ms L)$-martingale problem starting at $x\in S$.
\end{lemma}
\begin{proof}
Fix some $x=1/k$ and let ${\ms P}$ be a probability satisfying the assumption. Consider the $(\mc G_t)$-stopping time
$$
\tau\;:=\; \min \{ t\ge 0 : X_t < 1/k\} \;.
$$
Since
\begin{equation*}
\int_0^{t\land \tau} \ms Lf(X_s)ds \;=\; \int_0^{t} \ms L^k f(X_{s\land\tau})ds \;,\quad \forall t\ge 0\;,
\end{equation*}
then
\begin{equation*}
f(X_{t\land \tau}) - \int_{0}^{t} \ms L^k f(X_{s\land\tau})ds \;,\quad t\ge 0
\end{equation*}
is a $\ms P$-martingale, for any $f\in \dom$. Equivalently, if $X^{\tau} : D(\bb R_+,S) \to D(\bb R_+,S)$ denotes the measurable map defined by
\begin{equation*}
X_t \circ X^{\tau} \;=\; X_{t\land \tau} \;,\quad \forall t\ge 0
\end{equation*}
then the law of $X^{\tau}$ under $\ms P$, denoted by $\ms P\circ (X^{\tau})^{-1}$, turns out to be a solution of the $(\dom, \ms L^k)$-martingale problem. By Remark \ref{obs9} we conclude that
\begin{equation}\label{al1}
\ms P\circ (X^{\tau})^{-1} \;=\; \ms P_{1/k} \;,
\end{equation}
which in turn implies that
\begin{equation*}
\ms P\big(X_{t\land\tau}\ge 1/k \,,\; \forall t\ge 0\big) \;= \; \ms P_{1/k}\big(X_{t}\ge 1/k \,,\; \forall t\ge 0 \big)
\end{equation*}
Since the right hand side above equals one, then $\ms P(\tau=\infty)=1$ and so 
\begin{equation}\label{al2}
\ms P\circ (X^{\tau})^{-1} \;=\; \ms P \;.
\end{equation}
The desired result follows from \eqref{al1} and \eqref{al2}.
\end{proof}

\subsection{A strong Markov property}
As our next step, we prove Lemma \ref{lmp} below which relates any solution of the $(\dom, \ms L)$-martingale problem with laws $\{\ms P_{x}\}_{x\in S\setminus \{0\}}$ we just introduced. 

Let $\vartheta: \bb R_+\times D(\bb R_+,S) \to D(\bb R_+,S)$ be the measurable map defined by
\begin{equation*}
X_t\circ \vartheta(s,\cdot) \;=\; X_{s+t}(\cdot) \;,\quad \textrm{for all $t,s\ge 0$} \;.
\end{equation*}
In addition, given any $(\mc G_t)$-stopping time $\tau$ we define $\vartheta_{\tau}:D(\bb R_+,S)\to D(\bb R_+,S)$ as
\begin{equation*}
\vartheta_{\tau}(\omega) \;:=\;
\begin{cases}
\vartheta(\tau(\omega),\omega)\,, & \textrm{if $\tau(\omega)<\infty$}\,, \\
\omega\,, & \textrm{otherwise}\;.
\end{cases}
\end{equation*}

Consider the system of neighborhoods of $0\in S$
\begin{equation*}
A_k := \{ x\in S : x < 1/k \}\;, \quad k\in \bb N \;,
\end{equation*}
and their corresponding exit times
\begin{equation*}
\sigma_k \;:=\; \inf\{ t\ge 0 : X_t \in S\setminus A_k \} \;, \quad k\in \bb N \;.
\end{equation*}
Since $A_k$ and $S\setminus A_k$ are closed subsets then every $\sigma_k$ is a stopping time and 
\begin{equation}\label{obs12}
X_{\sigma_k}\;\ge \; 1/k \quad {\rm on} \quad \{\sigma_k<\infty\}\;.
\end{equation}

\begin{lemma}\label{lmp}
Let $\ms P$ be any solution of the $(\dom,\ms L)$-martingale problem and let $k\in \bb N$. For any $\mc C\in \mc G_{\infty}$, we have
\begin{equation}\label{time1}
\ms P \{\vartheta_{\sigma_k} \in \mc C \,,\; \sigma_k<\infty \} \;=\; \int_{\{\sigma_k<\infty\}}  \ms P_{X_{\sigma_k}(\omega)}(\mc C) \,\ms P(d\omega) \;.
\end{equation}
(Recall observation \eqref{obs12}.)
\end{lemma}
\begin{proof}
Fix $k\in \bb N$ and let $\{\ms Q_{\omega} : \omega \in D(\bb R_+,S)\}$ be a conditional probability distribution of $\ms P$ given $\mc G_{\sigma_k}$ such that for all $\omega \in D(\bb R_+,S)$ we have
\begin{equation}\label{reg6}
\ms Q_{\omega}(\mc A) \;=\; \delta_{\omega}(\mc A) \;,\quad \forall \mc A\in \mc G_{\sigma_k}\;.
\end{equation}
The existence of such $\{\ms Q_{\omega}\}$ is established in \cite[{Theorem 1.3.4}]{sv79} for a space of continuous paths but the same proof apply for $D(\bb R_+,S)$. Taking conditional expectation with respect to $\mc G_{\sigma_k}$ in the left hand side below we have
\begin{equation*}
\ms P\{\vartheta_{\sigma_k} \in \mc C\,, \; \sigma_k<\infty \} \;=\; \int_{\{\sigma_k<\infty\}} \ms Q_{\omega}\{\vartheta_{\sigma_k} \in \mc C\} \, \ms P(d\omega) \;.
\end{equation*}
Applying \eqref{reg6} we get $\ms Q_{\omega}\{\sigma_k= \sigma_k(\omega)\}=1$ for all $\omega$ and so the right hand side above equals
\begin{equation}\label{sq0}
\int_{\{\sigma_k<\infty\}} \ms Q_{\omega}\{\vartheta_{\sigma_k(\omega)} \in \mc C\}\, \ms P(d\omega) \;.
\end{equation}
Now, we relate $\{\ms Q_{\omega}\}$ to $\{\ms P_{x}\}_{x\in S\setminus \{0\}}$. For each $f\in \ms D_0$, we know that the process $(M^f_t)$
defined in \eqref{mm1} is a $\ms P$-martingale. Then, in virtue of \cite[Theorem 1.2.10]{sv79}, for each $f\in {\ms D}_0$ there exists some $\mc A_f \in \mc G_{\sigma_k}$ with $\ms P[\mc A_f]=1$ such that, for all $\omega\in \mc A_f\cap \{\sigma_k<\infty\}$,
\begin{equation}\label{sq1}
\textrm{$(M^f_t)$ is a $\ms Q_{\omega}$-martingale after time $\sigma_k(\omega)$},
\end{equation}
i.e. $\ms Q_{\omega}[M^f_{t_2} | \mc G_{t_1}] \stackrel{\mbox{\tiny $\ms Q_{\omega}$-a.s}}{=} M^f_{t_1}$, whenever $\sigma_k(\omega)\le t_1 < t_2$, where $\ms Q_{\omega}[\,\cdot \,|\, \cdot\,]$ stands for conditional expectation with repect to $\ms Q_{\omega}$. It follows from \eqref{sq1} that,
\begin{equation}\label{sq2}
\textrm{$(M^f_t)$ is a $\ms Q_{\omega}\circ (\vartheta_{\sigma_k(\omega)})^{-1}$-martingale }.
\end{equation}
Let us consider the countable subset of $\dom$
\begin{equation*}
\tilde{\ms D}_0 \;:=\; \big\{ f\in \ms D_0 : \textrm{$f(x)$ is a rational number for all $x\in S$} \big\}
\end{equation*}
and denote $\mc A := \bigcap_{f\in \tilde{\ms D}_0} \mc A_f$. Then, \eqref{sq2} implies that, for all $\omega\in \mc A\cap \{\sigma_k<\infty\}$,
\begin{equation*}
\textrm{$\ms Q_{\omega}\circ (\vartheta_{\sigma_k(\omega)})^{-1}$ is a solution of the $(\tilde{\ms D}_0,\ms L)$-martingale problem.}
\end{equation*}
But, given any $f\in \ms D_0$, $\exists$ $(f_n)$ in $\tilde{\ms D}_0$ such that $f_n\to f$ and $\ms Lf_n\to \ms Lf$, both pointwise, and
such that 
\begin{equation*}
\sup_{n\ge 1} \max_{x\in S} \big( |f_n(x)| + |\ms L f_n(x)| \big) <\infty \;.
\end{equation*}
By using this approximation it is easy to conclude that, for all $\omega\in \mc A\cap \{\sigma_k<\infty\}$,
\begin{equation}\label{sq3}
\textrm{$\ms Q_{\omega}\circ (\vartheta_{\sigma_k(\omega)})^{-1}$ is a solution of the $(\ms D_0,\ms L)$-martingale problem.}
\end{equation}
On the other hand, for all $\omega\in \{\sigma_k<\infty\}$,
\begin{equation*}
\ms Q_{\omega}\circ (\vartheta_{\sigma_k(\omega)})^{-1}\{X_0=X_{\sigma_k}(\omega)\} \;=\; \ms Q_{\omega}\{X_{\sigma_k(\omega)}=X_{\sigma_k}(\omega)\} \;=\; 1
\end{equation*}
(we applied \eqref{reg6} in the last equality.) Namely, for all $\omega\in \{\sigma_k<\infty\}$,
\begin{equation}\label{sq4}
\textrm{$\ms Q_{\omega}\circ (\vartheta_{\sigma_k(\omega)})^{-1}$ is starting at $X_{\sigma_k}(\omega)\in S\setminus \{0\}$\;,}
\end{equation}
where we used observation \eqref{obs12} for the last assertion. We may now conclude from \eqref{sq3}, \eqref{sq4} and the uniqueness result  established in Lemma \ref{unik} that
\begin{equation*}
\ms Q_{\omega}\circ (\vartheta_{\sigma_k(\omega)})^{-1} \;=\; \ms P_{X_{\sigma_k}(\omega)} \;,\quad \forall \omega \in \mc A \cap \{\sigma_k<\infty\}\;.
\end{equation*}
Since $\ms P(\mc A)=1$, this last assertion implies that \eqref{sq0} equals
\begin{equation*}
\int_{\{\sigma_k<\infty\}}  \ms P_{X_{\sigma_k}(\omega)}(\,\mc C\,) \,\ms P(d\omega) \;.
\end{equation*}
This concludes the proof.
\end{proof}

\subsection{A solution starting at $0\in S$}
From now on, we shall denote by $\ms P_0$ the law of $(\ms X_t)$ (defined in \eqref{xx}) so that we have now the \emph{complete} set of laws $\{\ms P_x : x\in S\}$. Obviously $\ms P_0$ starts at $0$. We prove now that $\ms P_0$ is a solution of the $(C^1(S),\ms L)$-martingale problem. Recall the sequence $(T_n)_{n\ge 2}$ of independent random variables considered in \eqref{nec}. For each $k\in \bb N$ define the process $(\ms X^k_t)$ as
\begin{equation*}
\ms X^k_t \;=\;
\begin{cases}
1/k\,, & 0\le t < T_k\,, \\
1/(k-1)\,, & T_k \le t < T_k + T_{k-1}\,, \\
\vdots & \vdots \\
1/2\,, & \sum_{n=3}^k T_n \le t < \sum_{n=2}^k T_n\,,\\
1 \,, & t\ge \sum_{n=2}^k T_n \;,
\end{cases}
\end{equation*}
for all $t\ge 0$. Clearly, the law of $(\ms X^k_t)$ is $\ms P_{1/k}$. Also, observe that $(\ms X^k_t)$ is related to $(\ms X_t)$ by
\begin{equation*}
\ms X^k_t \;=\; \ms X_{S_k + t} \;,\quad \forall t\ge 0 \;, \quad \textrm{where\; $S_k:=\sum_{n=k+1}^{\infty} T_n$} \;.
\end{equation*}
In particular, for all $t\ge 0$,
\begin{equation}\label{cc}
f(\ms X^k_t) \xrightarrow{a.s.} f(\ms X_{t}) \quad {\rm and} \quad \ms Lf(\ms X^k_t) \xrightarrow{a.s.} \ms Lf(\ms X_{t}) \;, \quad \textrm{as $k\uparrow \infty$} \;.
\end{equation}
Fix an arbitrary $f\in C^1(S)$, a continuous function $G:S^m \to \bb R$ and a finite set of times $0\le s_1 < \cdots < s_m \le s < t$. In virtue of Remark \ref{obs10}, we have
\begin{equation}\label{em1}
E\Big[ G(\ms X^k_{s_1},\dots,\ms X^k_{s_{m}}) \big\{ f(\ms X^k_t) - f(\ms X^k_s) - \int_s^t \ms Lf(\ms X^k_r)dr \big\} \Big] \;=\; 0 \;,
\end{equation}
for all $k\ge 1$. Letting $k\uparrow \infty$ in \eqref{em1} and using \eqref{cc} we get
\begin{equation}
E\Big[ G(\ms X_{s_1},\dots,\ms X_{s_{m}}) \big\{ f(\ms X_t) - f(\ms X_s) - \int_s^t \ms Lf(\ms X_r)dr \big\} \Big] \;=\; 0 \;.
\end{equation}
We have thus shown that $\ms P_0$ is a solution of the $(C^1(S),\ms L)$-martingale problem.

\subsection{Uniqueness starting at $0\in S$}
In this subsection we prove the uniqueness result that we used in Section \ref{teo12}. Let $\sigma$ stand for the exit time from $0\in S$, i.e.
\begin{equation}\label{esig1}
\sigma \;:=\; \inf\{ t\ge 0 : X_t \not = 0 \}\;.
\end{equation}
Clearly, $\sigma_k\downarrow \sigma$ pointwise. Notice that $\sigma$ is not a $(\mc G_t)$-stopping time.

\begin{proposition}\label{uni0}
There exists a unique probability measure $\ms P$ on $(D(\bb R_+,S),\mc G_{\infty})$ such that $\ms P\{X_0=0,\sigma=0\} = 1$ and 
$$
f(X_t) - \int_0^t \ms L f(X_s)ds\;, \quad t \ge 0
$$
is a $\ms P$-martingale for every $f\in \ms D_0$.
\end{proposition}

Existence is, of course, a consequence of Lemma \ref{lem15}. Nevertheless, it follows from the conclusion of the previous subsection that $\ms P_0$ fulfils all the requirements. In order to show uniqueness we first improve the result obtained in Lemma \ref{lmp}.

\begin{proposition}\label{lq2}
Let $\ms P$ be a solution of the $(\dom,\ms L)$-martingale problem starting at $0\in S$. If $\ms P\{\sigma=0\}=1$ then
\begin{equation*}
\ms P\{ \vartheta_{\sigma_k}\in \mc C \} \;=\; \ms P_{1/k}(\mc C) \;,\quad \textrm{$\forall$ $k\ge 1$ and $\mc C\in \mc G_{\infty}$}\;.
\end{equation*}

\end{proposition}
\begin{proof}
We start showing that
\begin{equation}\label{ita}
\ms P \{ \sigma_m < \infty \,,\; \forall m\in \bb N \} \;=\; 1 \;.
\end{equation}
Let us denote
\begin{equation*}
\mc A \;:=\; \{\sigma_m < \infty\,,\;\forall m\in \bb N\} \;=\; \{ \sigma_1 < \infty \} \;.
\end{equation*}
Since $\ms P_{1/n}(\mc A)=1$ for any $n\in \bb N$ then applying equation \eqref{time1} for $\mc C=\mc A$ and using observation \eqref{obs12} we get
\begin{equation*}
\ms P \{ \vartheta_{\sigma_k} \in \mc A \,,\; \sigma_k<\infty \} \;=\; \ms P\{\sigma_k<\infty \} \;, \quad \forall k\in \bb N \;.
\end{equation*}
But $\sigma_k + \sigma_1\circ \vartheta_{\sigma_k} = \sigma_1$ and so $ \{\vartheta_{\sigma_k} \in \mc A \,,\; \sigma_k<\infty \}=\mc A$. Using this observation in the last displayed equation we get
\begin{equation*}
\ms P ( \mc A ) \;=\; \ms P\{\sigma_k<\infty \} \;, \quad \forall k\in \bb N \;.
\end{equation*}
Since $\{\sigma_k<\infty\}\uparrow \{\sigma<\infty\}$ then, letting $k\uparrow \infty$ in the previous equation, we get $\ms P(\mc A)=\ms P\{\sigma <\infty\}$ which equals one by assumption.

As second step, we prove that
\begin{equation}\label{itb}
\ms P\big\{ X_{\sigma_m}=1/m \,,\; \forall m\in \bb N \big\} \;=\; 1 \;.
\end{equation}
For it, consider the events
\begin{equation*}
\mc B_n \;:=\; \{ X_0 =1/n \;\;{\rm and}\;\; X_{\sigma_m}=1/m \;\; \textrm{for all}\; 1\le m\le n \} \;,\quad n\in \bb N
\end{equation*}
and $\mc B := \bigcup_{n\in \bb N} \mc B_n$. Since $\ms P_{1/n}(\mc B_n)=1$ for all $n\ge 1$, then, for all $k\in \bb N$, we have
\begin{equation*}
\ms P_{X_{\sigma_k}(\omega)}(\mc B) \;=\; 1 \;, \quad \forall \omega \in \{\sigma_k < \infty\} \;.
\end{equation*}
Applying \eqref{time1} for $\mc C=\mc B$ along with this last observation we get
\begin{equation*}
\ms P \{ \vartheta_{\sigma_k} \in \mc B \,,\; \sigma_k<\infty \} \;=\; \ms P\{\sigma_k<\infty\} \;=\; 1 \;, \quad \forall k\in \bb N \;.
\end{equation*}
We used \eqref{ita} in the last equality. Therefore,
\begin{equation}\label{pr1}
\ms P \big\{ \vartheta_{\sigma_k} \in \mc B \; \textrm{and} \; \sigma_k<\infty\,,\; \textrm{for all $k\ge 1$} \big\} \;=\;1 \;.
\end{equation}
Now \eqref{itb} follows from \eqref{pr1}, assumption $\ms P\{X_0=0\,,\;\sigma=0\}=1$ and the following observation
\begin{equation*}
\{X_0=0 \,,\; \sigma =0 \,,\; \forall k\ge 1\,,\; \vartheta_{\sigma_k} \in \mc B\,,\;\sigma_k<\infty \} \; \subseteq \; \{ X_{\sigma_m}=1/m \,,\; \forall m\in \bb N \}
\end{equation*}
To prove this inclusion, fix some $\omega$ in the event of the left hand side and fix an arbitrary $m'\in \bb N$. Since $\sigma_k(\omega)\downarrow \sigma(\omega) = 0$ then $X_{\sigma_k}(\omega) \to  X_0(\omega) = 0$ as $k\uparrow \infty$ and so
\begin{equation}\label{peq1}
\textrm{$\exists$ ${k'}\in \bb N$ such that $X_{\sigma_{{k'}}}(\omega)< 1/m'$} \;.
\end{equation}
On the other hand, $\vartheta_{\sigma_{k}}(\omega)\in \mc B$ for all $k\in \bb N$ and so $\exists$ $n' \in \bb N$ such that 
\begin{equation}\label{peq2}
\vartheta_{\sigma_{{k'}}}(\omega)\in \mc B_{n'} \;.
\end{equation}
In virtue of \eqref{peq1} and \eqref{peq2} we necessarily have
\begin{equation*}
m' \;<\; n' \;\le\; {k'}
\end{equation*}
because
\begin{equation*}
1/k' \;\stackrel{\mbox{\tiny\eqref{obs12}}}{\le} \; X_0 \circ \vartheta_{\sigma_{{k'}}}(\omega) \;\stackrel{\mbox{\tiny\eqref{peq2}}}{=}\; 1/n' \;=\;X_0 \circ \vartheta_{\sigma_{{k'}}}(\omega)\;\stackrel{\mbox{\tiny\eqref{peq1}}}{<}\; 1/m' \;.
\end{equation*}
From \eqref{peq2} it follows that
\begin{equation*}
X_{\sigma_m} \circ \vartheta_{\sigma_{k'}}(\omega) \;=\; 1/m\;, \quad \forall\; 1\le m \le n' \;.
\end{equation*}
Since $m'<n'$ in particular we have
\begin{equation*}
X_{\sigma_{m'}} \circ \vartheta_{\sigma_{k'}}(\omega) \;=\; 1/m'\;.
\end{equation*}
But $X_{\sigma_{m'}} \circ \vartheta_{\sigma_{k'}}(\omega)=X_{\sigma_{m'}}(\omega)$ since $m'<k'$ and so $X_{\sigma_{m'}}(\omega)=1/m'$. This concludes the proof of the desired inclusion. 

Finally, the desired result follows from \eqref{itb} and \eqref{time1}.
\end{proof}

\begin{proof}[Proof of Proposition \ref{uni0}]
Let $\ms P$ be a probability satisfying the stated assumptions and let $E$ and $E_{1/k}$ stand for expectation with respect to $\ms P$ and $\ms P_{1/k}$ respectively. Fix an arbitrary $n\in \bb N$ some $0\le t_1 < t_2< \cdots < t_n$ and a bounded continuous
function $F:S^n\to \bb R$. In virtue of \eqref{lq2} we have
\begin{equation*}
E \big[F(X_{\sigma_k + t_1},\dots,X_{\sigma_k + t_n})\big] 
\; = \;   E_{1/k}\big[ F(X_{t_1},\dots,X_{t_n}) \big] \;, \quad \forall k\in \bb N\;.
\end{equation*}
But $(X_{\sigma_k + t_1},\dots,X_{\sigma_k + t_n}) \to (X_{t_1},\dots,X_{t_n})$ $\ms P$-a.s. as $k\uparrow \infty$ and so
\begin{equation*}
E\big[ F(X_{t_1},\dots,X_{t_n}) \big] \;=\; 
\lim_{k\to \infty} E_{1/k}\big[ F(X_{t_1},\dots,X_{t_n}) \big] \;.
\end{equation*}
This guarantees the desired uniqueness.
\end{proof}

\subsection{Proof of Proposition \ref{prop1}} 
In virtue of Remark \ref{obs10} and Lemma \ref{unik}, in order to conclude the proof of Proposition \ref{prop1}, it remains to prove that $\ms P_0$ is the unique solution of the $(C^1(S),\ms L)$-martingale problem starting at $0\in S$.

Observe that $fg\in C^1(S)$ for all $f,g\in C^1(S)$. We shall make use of the \emph{carr\'e du champ} corresponding to $(C^1(S),\ms L)$:
\begin{equation*}
\Gamma(f,g) \;:=\; {\ms L}({fg}) - {g}\ms L {f} - {f}\ms L {g} \;, \quad \textrm{for every $f,g\in C^1(S)$} \;.
\end{equation*}
Clearly, $\Gamma(f,g)$ turns out to be continuous for each $f,g\in C^1(S)$. Since $\ms L$ acts as a derivation at $0\in S$ we have
\begin{equation}\label{deri}
\Gamma(f,g)(0)\;=\;0\;,\quad \forall f,g\in C^1(S) \;. 
\end{equation}
Recall definition of $(M^f_t)$ given in \eqref{mm1} for each $f\in C^1(S)$.

\begin{lemma}\label{comp}
Let $\ms P$ be any solution of the $(C^1(S),\ms L)$-martingale problem. For all $f,g\in C^1(S)$, the process
\begin{equation*}
M^f_t M^g_t - \int_0^t \Gamma(f,g)(X_s)\, ds \;,
\quad t\ge 0 \;,
\end{equation*}
is a $\ms P$-martingale with respect to $(\mc G_t)$.
\end{lemma}
\begin{proof}
Fix some $f,g\in C^1(S)$. Denote
\begin{equation*}
V^f_t \;:=\; \int_0^t \ms Lf(X_s) \, ds \quad {\rm and} \quad V^g_t \;:=\; \int_0^t \ms Lg(X_s) \, ds \;, \quad t\ge 0\;,
\end{equation*}
so that, for all $t\ge 0$,
\begin{equation*}
M^f_t + V^f_t \;=\; f(X_t) \quad {\rm and} \quad M^g_t + V^g_t \;=\; g(X_t)
\end{equation*}
By multiplying these equalities we get
\begin{equation}\label{mmu}
M^f_t M^g_t + V^f_t V^g_t + M^f_t V^g_t + V^f_t M^g_t \;=\; (fg)(X_t) \;.
\end{equation}
By using
\begin{equation*}
(fg)(X_t) \;=\; M^{fg}_t + \int_0^{t} \ms L(fg)(X_s) ds \;,\quad t\ge 0 \;,
\end{equation*}
along with
\begin{equation*}
V^f_t V^g_t \;=\; \int_0^t V^f_s d V^g_s + \int_0^t V^g_s d V^f_s \;,\quad t\ge 0 \;,
\end{equation*}
in equality \eqref{mmu} we get
\begin{multline}\label{mmu2}
M^f_t M^g_t + M^f_t V^g_t + V^f_t M^g_t \\ =\; M^{fg}_t + \int_0^{t} \ms L(fg)(X_s) ds - \int_0^t V^f_s d V^g_s - \int_0^t V^g_s d V^f_s  \;.
\end{multline}
If we denote, for all $t\ge 0$,
\begin{equation}\label{mart2}
M^1_t \;:=\; M^f_t V^g_t - \int_0^t M^f_s d V^g_s  \quad {\rm and} \quad M^2_t \;:=\; M^g_t V^f_t - \int_0^t M^g_s d V^f_s \;,
\end{equation}
then equality \eqref{mmu2} can be rewritten as
\begin{equation}\label{mmu3}
M^f_t M^g_t + M^1_t  + M^2_t  \; = \; M^{fg}_t + \int_0^t \Gamma(f,g)(X_s)ds  \;.
\end{equation}
By assumption, $(M^{fg}_t)$ is a $\ms P$-martingale. In addition, in virtue of \cite[Theorem 1.2.8]{sv79}, $(M^1_t)$ and $(M^2_t)$ are also $\ms P$-martingales. Therefore the desired result follows from \eqref{mmu3}.
\end{proof}

We now use observation \eqref{deri} to prove that $0\in S$ is an instantaneous state for any solution starting at $0$.
\begin{lemma}
\label{leh}
For any solution $\ms P$ of the $(C^1(S),\ms L)$-martingale problem
starting at $0\in S$ we have $\ms P \{\sigma = 0\} = 1$.
\end{lemma}

\begin{proof} Let $\ms P$ be a probability satisfying the assumptions. Define $f:S\to \bb R$ as the inclusion function i.e. $f(x)=x$, for $x\in S$. Clearly $f \in C^1(S)$ and so
\begin{equation}
\label{mar1}
M_t \;:=\; X_t - \int_0^t ({\ms L}{f})(X_s)\, ds \;,\quad t\ge 0
\end{equation}
is a $\ms P$-martingale. Since $\sigma_k$ is a stopping time then it follows from Lemma \ref{comp} that
\begin{equation*}
(M_{t\land\sigma_k})^2 - \int_0^{t\land\sigma_k} \Gamma(f,f)(X_s)\, ds \;,
\quad t\ge 0 \;,
\end{equation*}
is a $\ms P$-martingale. In particular, for all $t\ge 0$ we have
\begin{equation}
\label{mc1}
{E}\big[ (M_{t\land \sigma_k})^2 \big] 
\;=\; {E} \Big[ \int_0^{t\land \sigma_k} 
\Gamma(f,f)(X_s)\, ds \Big] \;, 
\quad  \forall k\in \bb N \;,
\end{equation}
(since $M_0=0$, $\ms P$-a.s.) where $E$ represents the expectation with respect to $\ms P$. By the bounded convergence theorem, letting $k\uparrow\infty$ in \eqref{mc1} we get
\begin{equation}\label{mc2}
{ E}\big[ (M_{t\land \sigma})^2 \big] 
\;=\; { E} \Big[ \int_0^{t\land \sigma} 
\Gamma(f,f)(X_s)\, ds \Big]\;,\quad \forall t\ge 0 \;.
\end{equation}
Since $\{s<\sigma\} \subseteq \{X_s=0\}$, the right hand side in the above equation equals
\begin{equation*}
{E} [ t\land \sigma ] \;\Gamma(f,f)(0)
\end{equation*}
which vanishes as noticed in observation \eqref{deri}. Therefore, from \eqref{mc2} we conclude that 
\begin{equation*}
\ms P[ M_{t\land \sigma} = 0\,,\; \forall t\ge 0 ] \;=\; 1 \;.
\end{equation*}
Using this fact in \eqref{mar1} we get that, $\ms P$-a.s.,
\begin{equation*}
X_{t\land \sigma} \;=\; 
\int_0^{t\land \sigma} ({\ms L}{f})(X_s)\, ds 
\;=\; \frac 12 \big( t\land \sigma  \big) 
\;,\quad \forall t\ge 0 \;.
\end{equation*}	
But, for any $t>0$, we have on $\{t< \sigma \}$ that
\begin{equation*}
X_{t\land \sigma} \;=\; X_t \;=\; 0 \;\not=\; \frac{1}{2} \big(t\land \sigma\big) \;.
\end{equation*}
Hence $\ms P \{t < \sigma\} = 0$, $\forall$ $t>0$ and we are done.
\end{proof}

It follows from Lemma \ref{leh} and Proposition \ref{uni0} that
$\ms P_0$ is the only solution of the $(C^1(S),\ms L)$-martingale
problem.

\end{document}